\documentclass[11.1pt]{article}
\usepackage[]{amsmath,amssymb,amsfonts,latexsym,amsthm,enumerate,fullpage}
\usepackage[utf8]{inputenc}
\usepackage{color}
\usepackage{amsmath,amssymb,amsthm}
\usepackage{graphicx}
\usepackage{tikz}
\usetikzlibrary{math,calc}
\usepackage{rotating,stackengine,scalerel}

\newtheorem{thm}{Theorem}[section]
\newtheorem{prop}[thm]{Proposition}

\newtheorem{lemma}[thm]{Lemma}
\newtheorem{cor}[thm]{Corollary}

\newtheorem{definition}[thm]{Definition}

\newtheorem{rmk}[thm]{Remark}

\newtheorem{ex}[thm]{Example}
\newtheorem{obs}[thm]{Observation}

\newcommand{\R}{\mathbb{R}}

\newcommand{\Z}{\mathbb Z}
\newcommand{\F}{\mathbb F}
\newcommand{\N}{\mathbb N}

\newcommand{\rk}{\rho}

\newcommand{\nr}{\parallel}
\newcommand{\m}{\text{m}}
\newcommand{\p}{\text{p}}
\newcommand{\Ch}{\mathcal{C}}

\newcommand{\dia}{\diamond}
\newcommand{\xyz}{\langle}
\newcommand{\cdia}{c^\diamond}
\newcommand{\cxyz}{c^\langle}
\newcommand{\epsdia}{\epsilon^\diamond}
\newcommand{\epsxyz}{\epsilon^\langle}
\newcommand{\epsqr}{\epsilon^\Box}
\newcommand{\csame}{{C'}}
\newcommand{\cdiff}{{C''}}

\newcommand{\csametwo}{{B'}}
\newcommand{\cdifftwo}{{B''}}
\newcommand{\Nlow}{N^{-}}
\newcommand{\NN}{N^{-}}
\newcommand{\Nmid}{N^{\text{mid}}}
\newcommand{\wedgeV}{\wedge\rightarrow\vee}
\newcommand{\NwedgeV}{N^{\wedge\rightarrow\vee}}
\newcommand{\sqr}{\Box}
\newcommand{\csqr}{c^\Box}
\newcommand{\one}{\mathbf{1}}
\newcommand{\smallest}{\star}
\newcommand{\hf}{\hat{f}}
\newcommand{\hg}{\hat{g}}
\newcommand{\EE}{E}

\newcommand{\wye}{\scalerel*{\stackengine{-1pt}{%
  \rotatebox[origin=c]{30}{\rule{10pt}{.9pt}}\kern-1pt%
  \rotatebox[origin=c]{-30}{\rule{10pt}{1.3pt}}}{%
  \rule{.9pt}{10pt}}{O}{c}{F}{F}{S}}{\Delta}}
\newcommand{\Rwye}{R^{\wye}}

\begin{document}
\title{Garland's Technique for Posets and High Dimensional Grassmannian Expanders}

\author{
Tali Kaufman
\footnote{Department of Computer Science, Bar-Ilan University, Ramat-Gan, 5290002, Israel, email:kaufmant@mit.edu}
\and
Ran J. Tessler
\footnote{Department of Mathematics, Weizmann Institute of Science, POB 26, Rehovot 7610001, Israel, email:ran.tessler@weizmann.ac.il }
}

\maketitle
\begin{abstract}
Local to global machinery plays an important role in the study of simplicial complexes, since the seminal work of Garland \cite{G} to our days.
In this work we develop a local to global machinery for general posets. We show that the high dimensional expansion notions and many recent expansion results have a generalization to posets. Examples are fast convergence of high dimensional random walks generalizing \cite{KO,AL}, an equivalence with a global random walk definition, generalizing \cite{DDFH} and a trickling down theorem, generalizing \cite{O}.

In particular, we show that some posets, such as the Grassmannian poset, exhibit qualitatively stronger trickling down effect than simplicial complexes.

Using these methods, and the novel idea of \emph{posetification} to Ramanujan complexes \cite{LSV1,LSV2}, we construct a constant degree expanding Grassmannian poset, and analyze its expansion. This it the first construction of such object, whose existence was conjectured in \cite{DDFH}.
\end{abstract}

\section{Introduction}

High dimensional expanders are at the focus of an intensive recent study. What is a high dimensional expansion phenomenon? We argue that the high dimensional expansion phenomenon is the situation when certain global properties of a high dimensional object, such as fast convergence of random walks, are determined by certain properties of its marginals, i.e., by local properties of its "links", which are "small" substructures that compose it and correspond to local neighborhoods of the global object. 

This philosophy, that the global behaviour of a simplicial complex is determined by the behaviour of its local links, was present in many recent researches in mathematics and computer science \cite{G,O,KKL,EK,KM,DK,KO,AL}. This yoga was initiated in the pioneering work of Garland \cite{G}, who used, in the language of this paper, expansion of links in simplicial complexes to deduce vanishing of cohomologies, and since then had found many applications. This local to global philosophy for simplicial complexes is by now well studied. Developing such a theory beyond simplicial complexes and exploiting its consequences to the study of high dimensional expansion of more general objects is the first main goal of this work.

A more general framework in which one can study high dimensional expansion is the framework of general posets as was suggested by \cite{DDFH, KMS2} etc. The most notable example of an expanding poset, which is not a simplicial complex, is the Grassmannian Poset. It was recently
studied in relation to proving the 2-to-1 games conjecture \cite{KMS1,DKKMS1,DKKMS2,KMS2} . However, in this more general setting what does it mean for the object to be a high dimensional expander? Can we still adopt the yoga of local to global behaviour in the general case?

The work of \cite{DDFH} defined high dimensional expansion as a \emph{global} property which roughly occurs when two certain related random walks defined on the poset (the up-down walk and the down-up walk) are "close" to each other with respect to a natural norm. \cite{DDFH} show, under some assumptions, that this global definition coincides with an earlier (two-sided) local definition in the case of simplicial complexes.

The focus of our work is to define high dimensional expansion for general posets as a local to global property.
Namely, we say that a general poset is a high dimensional expander, iff all its links, which are posets of lower dimensions, possess certain expansion properties.  We further show that in such a case there is a way to deduce global properties of the poset, such as fast convergence of random walks on the poset, by local expansion properties of its links.
This is the philosophy we advocate in this work:
For this we need to explain what do we mean by general posets; we should then define their links; and then we have to show that expanding links imply that many global properties of the poset are dictated by the local properties of its links.

A central challenge is to find the correct \emph{axioms} a poset should satisfy in order to have a certain property of interest. We do it in three levels of generality: structural axioms (see below in this section), axioms on weights which generalize the former axioms (see in the beginning of sections \ref{sec:Garland},~\ref{sec:2sided},~\ref{sec:trickl}) and an approximate setting which generalizes the previous one.
We show that with the correct axiomatizations the main expansion theorems of \cite{KO,DDFH,O,AL} generalize. These results can serve as a tool box for future works on the subject.

Among these generalizations, notable is the generalization of Oppenheim's Trickling down \cite{O}. In this generalization different posets exhibit qualitatively different behaviours. For some posets, such as the Grassmannian poset, the expansion \emph{improves} while going down the links.

Finally, we use the tools we develop in this work to construct
a family of expanding high dimensional bounded degree posets that are not simplicial complexes. This is the first construction of such an object, whose existence was previously conjectured in \cite{DDFH}. This construction is achieved by sparsifying the Grassmannian poset using a high dimensional expander.  

\subsection{Posets; Links, Random Walk operators and regularity properties}
We start by briefly recalling what are posets, and their basic properties (see also Section \ref{sec:posets}).

\subsubsection{Posets and Graded Posets}

A \emph{poset} $(P,<)$ is a set $P,$ together with a binary order relation $<.$ We say $b$ \emph{covers} $a$ if $a<b$ and there is no intermediate $c$ with $a<c<b.$ 
A subposet is a subset of a poset, endowed with the restriction of the binary relation $\leq.$

Two useful examples of posets are simplicial complexes and Grassmannian poset.
\\\textbf{The simplicial complex Poset.} For a given set $S,$ all its subsets form a poset with respect to the containment order $\subseteq.$ 
Any subposet of such a poset, for an arbitrary underlying set $S,$ is called a \emph{simplicial complex}.
\\\textbf{The Grassmannian Poset.} Let $V$ be a space over a field $\F,$ the collection of its subspaces forms a poset, again with respect to containment. The poset is finite when $\F$ is finite and $\text{dim}_\F(V)< \infty.$ 
A \emph{Grassmannian poset} is any subposet of this poset, for an arbitrary $V.$

A \emph{graded poset} is a triple $(P,<,\rk)$ such that $(P,<)$ is a poset, together with a \emph{rank function} (See Section \ref{subsec:Graded posets}). For simplicial complexes the rank of a set $A$ can be taken to be $\rk(A) = |A|-1$. For the  Grassmannian poset, we can also define a rank by putting $\rk(U)=\dim_\F(U)-1.$

We put $P(i)=\rk^{-1}(i),$ and $C^i=\R^{P(i)},$ the space of real functions on $P(i).$ We write $\one\in C^i$ for the constant function $1.$
The \emph{rank of $P$} is the maximal $d$ for which $P(d)\neq\emptyset.$
A graded poset is said to be \emph{pure} if there exists $d$ such that for every element $x\in P$ there exists $y\in P(d),~x\leq y.$ In this case $P(d)$ is the set of maximal elements. Throughout this work all graded posets we consider, unless specified differently, are assumed to be finite and pure.

\subsubsection{Weighted Posets and Weighted Random Walks}

A \emph{weighted graded poset} is a triple $(P,<,\rho,\m,\p)$ where $(P,<,\rho)$ is a graded poset, together with a \emph{weight function} $\m:P\to \R_+,$ and transition probabilities $\p:P\times P\to \R_{\geq 0}$ which satisfy some relations (for the exact definition, and for other definitions in this subsection, see Subsection \ref{subsec:Weighted graded posets}). A weight scheme in which the transition probabilities $\p_{x\to y}=\frac{1}{\#\{z|z~\text{covered by }x\}}$ 
is called \emph{standard}, and a graded poset with such a weight scheme is called \emph{standard graded (weighted) poset}.
The weight function endows $C^i$ with a natural inner-product $\langle\cdot,\cdot\rangle.$ The associated norm is denoted $\nr\cdot\nr.$

\paragraph{The Up, Down Operators and the associated Random walks} The data of weights and transition probabilities allows defining the \emph{Up} and \emph{Down} operators. The Up operator $U_k$ maps $C^k\to C^{k+1},$ while the down operator $D_k$ maps $C^k$ to $C^{k-1}.$ Both mappings are defined using the structure constants of the weighted poset.
%

Combining these two operators we obtain two random walks which are defined on the weighted posets:
The \emph{up-down} random walk operator, $M_k^+ = D_{k+1}U_k$ corresponds to the random walk on
$P(k)$ defined as follows: given a $x \in P(k)$, choose randomly (according to the parameters of the weighted poset) $y \in P(k+1)$ such that $y$ covers $x$ and then choose randomly (according to the parameters of the weighted poset again) $z \in P(k)$, that is covered by $y$. 

The \emph{down-up} random walk operator, $M_k^-= U_{k-1}D_{k}$ is similarly defined, only that now starting from $x \in P(k)$, we first choose randomly $z \in P(k-1)$ that is covered by $x$, and then choose randomly $y \in P(k)$ such that $y$ covers $z$.

We also define the $l$th \emph{Adjacency operator} $A_l:C^l\to C^l,$ which is a non lazy version of the up-down walk. We write $A=A_0,$ for the zeroth adjacency operator (also called the adjacency operator). We normalize these operators so that the largest eigenvalue is $1.$

\subsubsection{Links and induced weight functions}

Let $P$ be a poset, and $x\in P.$ The subposet $P_x$ made of $x$ and all elements $y> x,$ is called the \emph{link of $x$}. 
If $P$ is graded, then so is $P_x$, and the induced rank function $\rk_x$ is given by
\[\rk_x(y)=\rk(y)-\rk(x)-1.\] If $P$ is in addition weighted, then we can also induce the weight function and transition probabilities (for exact definitions see Section \ref{subsec:links}).
We write $C^i_x$ for the space of real functions on $P_x,$ and write $\langle-,-\rangle_x,\nr-\nr_x$ for the induced inner products and norms,
we can also define the operator $D_{x,i}, U_{x,i}$ as above, only with $\m_x,\p_x$ instead of $\m,\p.$
The \emph{localization} is the linear map from $C^i$ to $C^{i-\rk(x)-1}_x$ which maps $f$ to $f_x,$ defined by $f_x(y)=f(y)$ for $y\in P_x$.

\paragraph{Basic Localization}
The starting point of what is now known as Garland's technique, is the observation \cite{G} that the global inner products $\langle f,g\rangle,~\langle Df,Dg\rangle,~\langle Uf,Ug\rangle$, in the case of simplicial complexes, can be written as sums of local inner products over links.
In this work we investigate to which extent these results can be generalized to more general posets. 
\begin{prop}[Basic Localization Property] \label{prop:Gland_posets_1}
Let $P$ be a graded weighted poset of rank $d.$ Let $-1\leq k < l\leq d,$ and $f,g\in C^l.$ Then
\begin{enumerate}
\item\label{eq:localization_of_inner_prod}
$\langle f,g\rangle = \sum_{x\in P(k)}\m(x)\langle f_x,g_x\rangle_x.$
\item\label{eq:localization_of_D_i}
$\langle D_l f,D_l g\rangle = \sum_{x\in P(k)}\m(x)\langle D_{x,l-k-1}f_x,D_{x,l-k-1}g_x\rangle_x.$
\end{enumerate}
\end{prop}

\subsubsection{Additional localization properties.} We will prove various different "localization properties" for different operators under various assumptions: We will prove a localization property for $\langle U_l f,U_l g\rangle$, for posets that satisfy the \emph{Up-Localization Property} (See Definition \ref{def:ass_I} and Proposition \ref{prop:garland_posets_2}).  We will prove a localization property for $\langle Af,f\rangle$, where $A$ is the adjacency matrix, for standard posets which satisfy the \emph{Adjacency Localization property} (See Definition \ref{def:for_equiv_defs} and Proposition \ref{prop:dec_adj}). We will also prove localization properties with respect to different localized functions which will be useful for the trickling down phenomenon, for posets which satisfy the \emph{Tricking Localization property} (See Definition \ref{def:TL} and Proposition \ref{prop:trickling_localization1}). Each of these different localization properties together with assumptions on the expansion of links, will dictate different global properties of the posets.  

\subsubsection{Regularity properties of posets}\label{subsec:regularity_intro}
Instead of providing the rather technical assumptions on the weights and transition probabilities of the weighted posets which give rise to the different localizations, we describe regularity properties of the structure of posets which are important special cases of those localization assumptions, and are the motivation for them. For more details on the regularity properties, see Subsection \ref{subsec:axioms}. In the body of the article we shall describe the more general assumptions, and verify that they indeed generalize these structural properties. In Section \ref{sec:approx} we shall generalize further, by requiring only \emph{approximated versions} of the localization assumptions.

\paragraph{Lower regularity.} A graded poset $P$ is said to be \emph{lower regular at level $i$} for $i>-1,$ if there exists a constant $\Nlow_i,$ such that every $x\in P(i),$ covers exactly $\Nlow_i$ elements of $P(i-1).$
The poset is \emph{lower regular} if it is lower regular at level $i,$ for every $i>-1.$

\paragraph{Middle regularity.} A graded poset $P$ is said to be \emph{middle regular at level $i$}, $i>-1,$ if there exists a constant $\Nmid_i$, such that for any $x>z$ with $x\in P(i+1),z\in P(i-1)$ there are precisely $\Nmid_i$ elements $y\in P(i)$ which cover $z$ and are covered by $x.$ $P$ is \emph{middle regular} if it is middle regular at level $i$ for each $i>-1.$

\paragraph{$\wedgeV$ regularity.} $P$ is \emph{$\wedgeV$ regular at level $i$}, $i>-1,$ if there is a constant $\NwedgeV_i$, such that for any $y_1,y_2\in P(i)$ which are covered by an element $x\in P(i+1)$ there are precisely $\NwedgeV_i$ elements $z\in P(i-1)$ covered by both. $P$ is \emph{$\wedgeV$ regular} if it has this property at each at level $i>-1$.
\\\textbf{$\wye$ regularity.} $P$ is \emph{$\wye$ regular} if there exists a constant $\Rwye=\Rwye(P)$ such that for each $u\in P(2),~y_1\neq y_2\in P(0),$ with $u>y_1,y_2,$ there are exactly $\Rwye$ elements $z\in P(1)$ satisfying $y_1,y_2\lhd z\lhd u.$

A poset is \emph{regular} if it is lower, middle and $\wedgeV$ regular. A poset $P$ is \emph{$2-$skeleton regular} if it is lower regular at levels $1,2,$ middle regular at level $1,$ 
and $\wye$ regular.

One can also consider local regularity properties, which means that the links are also required to be regular, and in a uniform way (which may depend on the rank). For example
$P$ is locally $2-$skeleton regular if for all $s\in P(\leq d-3),$ $P_s$ is $2-$skeleton regular with regularity structure constants which depend only on the level of $s$.

\paragraph{Regularity of the simplicial complex poset and the Grassmannian poset.}
A simplicial complex
is regular and $2-$skeleton regular. The regularity constants are $\Nlow_i=i+1,$ $\Nmid_i=2,$ $\NwedgeV=1,$ $\Rwye=1.$
Also a Grassmannian poset over $\F_q$ is regular and $2-$skeleton regular, with constants $\Nlow_i=[i+1]_q,$ $\Nmid_i=q+1,$ $\NwedgeV=1,$ $\Rwye=1.$

\subsubsection{Definitions of expanding posets}
We will study different notions of expanding posets.
A poset is \emph{connected} if $M^+_0$ induces an irreducible Markov chain.

The following global definition of an expanding poset was given in \cite{DDFH}.
\begin{definition}(Eposet - Global expanding poset)
A poset $P$ of rank $d$ is a $\lambda$-global eposet if for all $1 \leq j\leq d-1$ there exist constants $r_j,\delta_j$ such that
\[\nr D_{j+1}U_j-\delta_jU_{j-1}D_j - r_j Id_{C^j} \nr\leq\lambda.\]
\end{definition}

In this work we suggest an alternative, local to global definition of an expanding poset, generalizing the local to global definition that was studied for simplicial complexes.
\begin{definition}(One sided local spectral expander)
Let $P$ be a standard, weighted, graded poset of rank $d.$
Then $P$ is called one-sided $\lambda$-local spectral expanding poset if $P$ and any link $P_x,$ for $x\in P(i),~i\leq d-2,$ are connected, and the non trivial eigenvalues of the adjacency matrix of the link $P_x$, for every $x \in P(i),~i\leq d-2,$ are upper bounded by $\lambda.$
\end{definition}

\begin{definition}(Two sided local spectral expander)
Let $P$ be a standard, weighted, graded poset of rank $d.$
Then $P$ is called two-sided $[\nu,\lambda]$-local spectral expanding poset if $P$ and any link $P_x,$ for $x\in P(i),~i\leq d-2$ are connected, and the non trivial eigenvalues of the adjacency matrix of the link $P_x$, for every $x \in P(i),~i\leq d-2,$ lie in $[\nu,\lambda]$ .
\end{definition}

\paragraph{Local spectral expansion in the non standard setting.} The current definition of local spectral expansion in posets assumes that the adjacency matrix is diagonalizable with real eigenvalues. Moreover, for most applications we need assume that this matrix is self adjoint. This happens automatically in the standard setting, but it does not need to happen more generally.
An alternative way to define the local spectral expansion of posets can be by using the spectral gap of $M^+_{x,0},~x\in P(\leq d-2),$ which is always self adjoint with respect to the natural inner product induced from the weights, and is positive semi definite.  Such a definition does not assume a standard weight scheme. It should be noted that analyzing the spectrum of $M^+$ is not equivalent to analyzing the spectrum of the adjacency matrix. It is equivalent when $P$ is standard and locally lower regular at level $1.$ Then the spectra of $M^+$ and of the adjacency matrix differ by some scaling and shifting.
In this work we chose to define local spectral expansion according to the spectra of the link adjacency matrices, but many of the tools we have developed apply also for the alternative choice.


\subsection{Local to global theorems for general posets}

For the simplicial complex poset the following three theorems form a cornerstone in the study of local to global properties, and have led to several recent breakthroughs in computer science, such as counting bases of matroids \cite{AKOV} and optimal mixing in Glauber dynamics \cite{AKO,CLV}.

\begin{itemize}
\item Random walks convergence from one-sided expanding links: One sided local expansion in all links imply fast convergence of all high dimensional random walks \cite{KO,AL}. 

\item Trickling-down Theorem: If all top most links expand, and the complex and its links are connected, then the global underlying graph of the complex expands \cite{O}.

\item Equivalence between global random walk convergence and  two-sided expanding links: All links are two-sided expanders iff at each degree the Up-Down walk operator is almost identical to Down-Up walk operator \cite{DDFH}.
\end{itemize}

In what follows we generalize the above theorems for general posets which satisfy the previously mentioned localization assumptions (or in particular the previously defined regularity properties) and whose links are sufficiently expanding. Not only do we obtain generalizations of the above local to global theorems for general posets, but we also observe that different posets behave qualitatively different under these generalization. The different behaviour might have desired effects. One notable example is the Trickling Down theorem. This theorem for simplicial complexes, that was known prior to our work, bounds the expansion of the global complex by the expansion of its links, but in general the global expansion tends to be inferior to the expansion of the links. We show that for general posets with certain regularity conditions, that occur e.g, in the Grassmannian poset, the expansion of the global poset is {\bf superior} to that of the links. This plays an important role in our construction of bounded degree expanding posets that are not simplicial.

\subsubsection{Fast Random walk convergence from one sided local expansion for posets}
Our first main theorem is that we get fast mixing of random walks on general posets from local {\bf one}-sided spectral expansion in links. This is the first result of mixing of random walks for general posets that relies only on one-sided local expansion in links. Previous works that discussed expanding posets defined them to be expanding using global properties of their random walks \cite{DDFH, AGT}. In particular these previous works relied on {\bf two} sided global expansion bounds.

The main tool for achieving this result is the \emph{Up-Localization property} (Definition \ref{def:ass_I}, Proposition \ref{prop:garland_posets_2}).
For simplicity we state the theorem for standard regular posets, although it holds for any poset (not even standard) satisfying the more general Up Localization property.

\begin{thm}\label{thm:RW-one-sided-informal} [For the formal, much more general, statement, see Theorem \ref{thm:5_3_5_4}]
Let $P$ be a standard regular poset which is one sided $\lambda$ local spectral expander with $\lambda$ small enough then:

\[\nr \langle M_k^+ f , f \rangle \nr \leq (\max_{j\leq k}\{a_{k,j}\}+ \lambda(\max_{j\leq k}\frac{\Nmid_j-1}{\Nmid_j})\sum_{i=0}^k b_{k,i})\nr f\nr^2,\]
with the constants 

\[a_{l,r}=1-\frac{\Nlow_r}{\Nlow_{l+1}}\prod_{j=r}^l\frac{\Nmid_j-1}{\NwedgeV_j},~~b_{l,r}=\frac{\Nlow_r\Nmid_r}{\Nlow_{l+1}\NwedgeV_r}\prod_{j=r+1}^l\frac{\Nmid_j-1}{\NwedgeV_j}.\]


\end{thm}

\paragraph{Implications for Random walks in simplicial complexes and in the Grassmannian poset.}

For one-sided $\lambda-$local expanding simplicial complexes we recover a result of \cite{KO}: In particular for any $f\in C^k,~f\perp\one$
\[\langle M^+_k f,f\rangle \leq (\frac{k+1}{k+2}+\frac{k+1}{2}\lambda)\nr f\nr,\]so that the largest eigenvalue of $M^+_k|_{C^k_0}$ is at most $\frac{k+1}{k+2}+\frac{k+1}{2}\lambda,$ which is \cite[Theorem 5.4]{KO}.

For one-sided $\lambda-$local expanding Grassmannian posets we get for any $f\in C^k,~f\perp\one$
\[\langle M^+_k f,f\rangle \leq (\frac{[k+1]_q}{[k+2]_q} + S(k,q)\lambda ) \nr f\nr.\] where
$S(k,q)=\sum_{i=0}^k\frac{q^{k-i+1}[i+1]_q}{[k+2]_q}.$
Note that as $q$ grows, $S(k,q)=k+1+O(\frac{1}{q})$, and $\frac{[k+1]_q}{[k+2]_q}\to \frac{1}{q}.$

\paragraph{Random walks from sequence of spectral gaps on links.}

In Section \ref{sec:2sided} we generalize the recent result of \cite{AL}, which bounds the second eigenvalue of $M^+_k$ by the sequence of bounds $\mu_i$ on the second eigenvalues of the adjacency matrices of level $i$ links, to general posets. Again, for the clarity of the introduction we state the result in the special case of standard regular posets.
\begin{thm}\label{thm:generalized_Alev_Lau-infrormal} [For the formal and more general statement, see Theorem \ref{thm:generalized_Alev_Lau}]
Let $P$ be a standard regular poset. Suppose that for all $i,$ the second eigenvalue of any adjacency matrix of a level $i$ link is bounded from above by $\mu_i.$ Then
\[\lambda_2(M^+_k)\leq 1-\prod_{j=-1}^{k-1}\frac{\Nlow_{j+2}-1}{\Nlow_{j+2}}\prod_{j=-1}^{k-1}(1-\mu_j).\]
\end{thm}

The Up Localization Property is more suited to working with the operators $M^+_k$ and their localizations, while the Adjacency Localization Property is more suited to working with the adjacency opertors and their localizations. Neither property is strictly stronger than the other (although the Up Localization property needs not to assume a standard weight system), and they cover somehow different scenarios, though the regular cases are covered by both.
In fact most of the results of this paper which are stated for adjacency operators have counterparts in terms of the Up-Down walks operators, and vice versa, although we will not show that here.

\subsubsection{Equivalence between global RW convergence and two-sided expanding links }

Our next main theorem is a generalization of theorem of \cite{DDFH} showing that for simplicial complexes global random walks expansion is equivalent to two-sided local spectral expansion in links. 

The equivalence between two-sided local spectral expanders and eposets proven in \cite{DDFH} was stated only for simplicial complexes. Now, that we have defined local spectral expanding posets, and developed general local to global machinery we can generalize the equivalence between two-sided spectral expanders and eposets to a more general setting.

Our main tool is the "Adjacency Localization property", and again we state the theorem in the special case of standard regular posets.

\begin{thm}\label{thm:equiv_for_2_sided-informal}[For a formal, more accurate statement, see Theorem \ref{thm:equiv_for_2_sided}]
Suppose $P$ is standard and regular poset then $P$ is two-sided $\lambda$-local-spectral expander iff it is $O(\lambda)$-eposet.
\end{thm}

\subsubsection{Trickling Down theorem from local expansion for posets}

In Section \ref{sec:trickl} we generalize the notable Trickling-Down theorem of Oppenheim \cite{O}. Finding the correct axiomatization that allows this generalization is more tricky than in the previous situations. But it results with higher gain. We show that different posets behave qualitatively different under the trickling down process, and we learn a surprising fact: Unlike the simplicial complex case, where the bound for expansion deteriorates as we go down the links, for posets with large local lower regularity, such as the Grassmanian poset, the global expansion is improving over the links.

The main tool used is the \emph{Tricking Localization property} (Definition \ref{def:TL}, Proposition \ref{prop:trickling_localization1}), which generalizes local-$2-$skeleton regularity (See Subsection \ref{subsec:axioms}). For posets having this property one can bound the global underlying expansion using the expansion of the top links. Again we state the theorem in this section under simplifying assumptions, we also write only the "one-step" version of the theorem, leaving the full theorem with the repetitive application to Section \ref{sec:trickl}.

\begin{thm}\label{thm:trickling_structural-informal}[For the formal statement see Theorem \ref{thm:trickling_general}]
Let $P$ be a standard graded weighted poset. Suppose that $P$ is $2-$skeleton regular, with constants $\Nlow_1,\Nlow_2,$ $\Nmid_1$ and $\Rwye.$ Assume in addition that $P$ and any link $P_x,$ for $x\in P(0),$ are connected, and that the non trivial eigenvalues of the adjacency matrix of the link $P_x$  lie in $[\nu,\mu]$.
Then
\[\frac{\frac{\nu}{\Nlow_1-1}-\frac{({\Nlow_2\Nlow_1-\Nmid_1})\Rwye(\Rwye-1)}
{(\Nmid_1-1)(\Nmid_1)^2({\Nlow_1}-1)^2}}{1-\nu}\leq\lambda\leq\frac{\frac{\mu}{\Nlow_1-1}-\frac{({\Nlow_2\Nlow_1-\Nmid_1})\Rwye(\Rwye-1)}
{(\Nmid_1-1)(\Nmid_1)^2({\Nlow_1}-1)^2}}{1-\mu}.\]
\end{thm}

\paragraph{Implications for Trickling down for simplicial complexes and for Grassmannian poset.}

When $P$ is a simplicial complex we obtain the following upper and lower bound on the non trivial eigenvalues of the adjacency matrix:
$\frac{\nu}{1-\nu} \leq \lambda\leq \frac{\mu}{1-\mu}$, reproducing the result of \cite{O}.
Moving to the Grassmannian poset, we obtain the following upper and lower bound on the non trivial eigenvalues of the adjacency matrix:
$\frac{\nu}{q(1-\nu)}\leq \lambda\leq \frac{\mu}{q(1-\mu)}.$

The map $x\mapsto \frac{x}{q(1-x)},$ for $q>1,$ has two fixed points, $0$ and $\frac{q-1}{q}.$ The former is attractive and the latter is repulsive. Thus, if we have a rank $d$ locally connected Grassmannian poset, whose $d-2$ links have all their nontrivial eigenvalues in $[\nu,\mu],$ then if $\mu<\frac{q-1}{q},$ the non trivial eigenvalues become smaller in absolute value as we consider links of elements of lower and lower ranks. For simplicial complexes such phenomenon existed only for the negative eigenvalues. These ideas, and  $\frac{q-1}{q}$ being a critical value for the trickling process, will play a role in the analysis of Section \ref{sec:posetification} below.

\subsection{A construction of non simplicial bounded degree expanding posets}

Recall that the Ramanujan complexes \cite{LSV1, LSV2} are bounded degree expanding simplicial complexes. They can be seen as a sparsification of the complete simplex. It was asked by \cite{DDFH} whether a similar sparsification exists for the Grassmannian poset, and can the expansion parameter be arbitrary small. Namely, whether it is possible to construct a subposet of the Grassmannian that is expanding and is of bounded degree.

In Section \ref{sec:posetification} we provide an explicit sparsification of the complete Grassmannian, via a process which we call \emph{posetification}. We prove, using the tools developed in this work, that it is indeed expanding, with expansion parameters of magnitude $\frac{q-1}{q},$ partially answering the question of \cite{DDFH}. 
To the best of our knowledge, this is the first proof of existence, and the first known construction, of a bounded degree expanding subposet of the complete Grassmannian.

The posetification process which we describe below uses the sparsification for simplicial complexes obtained by a Ramanujan complex, to sparsify the the complete Grassmannian. The analysis of expansion relies heavily on the general trickling down we develop, and on the fact it amplifies expansion when one goes down the links, as opposed to the simplicial case. We then use the other tools we developed to bound the spectra of random walk operators on it.

\begin{thm} [Informal, For formal see Corollary \ref{cor:contruction}]
Let $q$ be a prime power and $d$ a natural number. Let $X$ be a $d-$dimensional Ramanujan complex of thickness\footnote{the thickness of the Ramanujan complex is the number of top cells a given codimension $1$ cell touches, minus $1$} $Q,$ where $Q$ is large enough as a function of $q,d.$ Then the posetification $V_X$ is a bounded degree global expanding, and a two-sided local spectral expanding Grassmannian poset. The second largest eigenvalue of each link is at most $\frac{q-1}{q}+o(1),$ while the lowest eigenvalue is at least $-\frac{1}{q}.$
\end{thm}

\subsection{Plan of the paper}
Section \ref{sec:posets} provides basic definitions for graded weighted posets, including those of links, localizations and regularity properties, and proves basic localization results a-la-Garland \cite{G}. Section \ref{sec:Garland} describes Property UL and its consequences, in particular a generalization of a decomposition theorem from \cite{KO} for one sided local spectral expanders. Section \ref{sec:2sided} describes property AL and its consequences, a generalization of the main theorems of \cite{DDFH,AL}. Section \ref{sec:trickl} generalizes Oppenheim's trickling down theorem \cite{O}.
Each of the sections \ref{sec:Garland}, \ref{sec:2sided} and \ref{sec:trickl} is structured so that first a property for posets is presented (UL,AL,TL), then it is shown which type of localization this property gives, and that it generalizes the corresponding regularity property. Finally the consequences of this localization is presented.
In Section \ref{sec:posetification} we construct a bounded degree expanding Grassmannian poset. Finally, in Section \ref{sec:approx} we briefly discuss the scenario where instead of having one of the assumptions UL, AL or TL, we only have an approximated version of it. We illustrate how in that case the theorems we proved in earlier sections generalize to approximate versions, with bounded error terms.

\subsection{Acknowledgments}
T.~K. was supported by ERC. R.~T. (incumbent of the Lillian and George Lyttle Career Development Chair) was supported by the ISF grant No. 335/19 and by a research grant from the Center for New Scientists of Weizmann Institute.

\section{Weighted graded posets: definitions, regularity properties and the basic decomposition}\label{sec:posets}
We begin by reviewing the standard definition of partially ordered set (poset), and the slightly less standard notions of graded poset, and weighted graded posets.
We then describe some additional properties that many posets share, and will play a role in our study.
\subsection{Posets}
A \emph{poset} $(P,\leq)$ is a set $P$ together with a binary relation $\leq$ which satisfies
\begin{enumerate}
\item\textit{Reflexivity} $\forall a\in P,~a\leq a.$
\item\textit{Antisymmetry} If $a\leq b$ and $b\leq a$ then $a=b.$
\item\textit{Transitivity} If $a\leq b$ and $b\leq c$ then $a\leq c.$
\end{enumerate}
We write $a<b$ ($a$ is strictly less that $b$) if $a\leq b$ and $a\neq b.$
$a$ is \emph{covered} by $b$ (equivalently, $b$ covers $a$) if $a<b$ but there is no intermediate $c$ with $a<c<b.$ In this case we write $a \lhd b.$ We can similarly define the relations $\geq,>,\rhd.$
For $z\in P$ we write $\NN(z)$ for the number of elements covered by $z.$
A \emph{minimal element for a set of elements} $\{a_\alpha\}_\alpha$ is an element $b$ in the set, which is smaller or equal all other elements in the set. A \emph{minimal element} is a minimal element for all the elements in the poset. \emph{Maximal elements} are similarly defined.

The poset is finite if the underlying set is. A subposet is a subset of a poset, endowed with the restriction of the binary relation $\leq.$

A \emph{chain} from $y$ to $x$ is a tuple $c=(c_1,\ldots,c_m)$ where each $c_i\in P,~c_1=x,c_m=y$ and $c_i<c_{i+1}$ for all $i.$ A chain is maximal if for all $i,$ $c_i\lhd c_{i+1}$. $\Ch_{y\to x}$ denotes the collection of chains from $y$ to $x.$
\begin{ex}\label{ex:1}
Examples of posets are the $\N,\Z,\R$ together with the standard orders.
The former has a unique minimal element $1.$ The others do not have. Neither have a maximal element.

For a given set $S,$ all its subsets form a poset with respect to the containment order $\subseteq.$ This poset is finite precisely when $|S|<\infty.$
The unique minimal element here is $\emptyset,$ and the unique maximal is $S.$
Any subposet of such a poset, for an arbitrary underlying set $S$ is called a \emph{simplicial complex}. The subposet obtained by restricting to subsets of size at most $k$ has the sets of size $k$ as maximal elements.

Let $M$ be a matroid, then its independent sets form a poset with respect to inclusion. This poset has the empty set as the unique minimal element, maximal elements are the bases, but usually there is no unique maximal element. This poset can naturally be considered as a simplicial complex whose underlying set is (the self independent) elements of $M,$ and whose simplices correspond to independent sets.

Let $V$ be a space over a field $\F,$ the collection of its subspaces forms a poset, again with respect to containment. The poset is finite when $\F$ is finite and $\text{dim}_\F(V)< \infty.$ The unique minimal element here is the $0$ vector space, and the unique maximal is $V.$
A \emph{Grassmannian poset} is any subposet of this poset, for an arbitrary $V.$
The subposet obtained by restricting to sub vector spaces of dimension at most $k$ has the vector subspaces of dimension $k$ as maximal elements.
\end{ex}
\subsection{Graded posets}\label{subsec:Graded posets}
A \emph{graded poset} is a triple $(P,\leq,\rk)$ such that $(P,\leq)$ is a poset, together with a \emph{rank function} $\rk: P\to\Z_{\geq -1}$ which is an order preserving map.
We further impose that
\begin{enumerate}
\item If $a\lhd b$ then $\rk(b)=\rk(a)+1.$
\item There is a single element $\smallest\in\rk^{-1}(-1).$
\end{enumerate}
We put $P(i)=\rk^{-1}(i),$ and $C^i=\R^{P(i)},$ the space of real functions on $P(i).$ We write $\one\in C^i$ for the constant function $1.$
The \emph{rank of $P$} is the maximal $d$ for which $P(d)\neq\emptyset.$
A graded poset is \emph{pure} if there exists $d$ such that for every element $x\in P$ there exists $y\in P(d),~x\leq y.$ In this case $P(d)$ is the set of maximal elements. The \emph{$i$th skeleton}, denoted $P(\leq i),$ is the subposet $\rho^{-1}([-1,i]).$
Throughout this work, all graded posets we consider are finite and pure.
\begin{ex}\label{ex:2}
Returning to Example \ref{ex:1}, for a given set $S$ define a rank function $\rk:2^S\to\Z_{\geq -1}$ by putting $\rk(A)=|A|-1.$
If $S$ is a matroid, we may use the same rank function, restricted to independent sets. For the vector spaces example we can define a rank by putting $\rk(U)=\dim_\F(U)-1.$
\end{ex}
\subsection{Weighted graded posets}\label{subsec:Weighted graded posets}
We now turn to consider weighted posets, following \cite{DDFH}. A \emph{weighted poset} is a triple $(P,\leq,\m,\p)$ where $(P,\leq)$ is a poset, together with a \emph{weight function} $\m:P\to \R_+,$ and transition probabilities $\p:P\times P\to \R_{\geq 0}$ which satisfy
\begin{enumerate}
\item $\p_{y\to x}>0$ if and only if $y\rhd x.$
\item $\forall y,~\sum_{x\lhd y}\p_{y\to x}=1.$
\item For any $x\in P$ which is not maximal,
\begin{equation}\label{eq:weight}
\m(x)=\sum_{y\rhd x}\p_{y\to x}\m(y).
\end{equation}
\end{enumerate}
Thus, $\m$ is determined by its values on maximal elements and the transition probabilities.
In particular, any choice of weights $\m$ for maximal elements of $P,$ can be extended to a weight function $\m$ satisfying
\[\m(x)=\sum_{y\rhd x}\frac{\m(y)}{\NN(y)},\]
i.e. all transition probabilities from $y$ equal $\frac{1}{\NN(y)}$. We call such a weight scheme a {standard weight scheme},
and the weighted poset is called a \emph{standard weighted poset}.

We define $U:\R^P\to \R^P$ to be the Markovian transition operator,\[\forall g\in \R^P,~y\in P,~ (Ug)(y)=\sum_{x\lhd y}\p_{y\to x}g(x).\] We similarly define $D:\R^P\to \R^P$
\[\forall f\in \R^P,~x\in P,~(Df)(x)=\sum_{y\rhd x}\frac{\p_{y\to x}\m(y)}{\m(x)}f(y).\]
For a maximal chain $c=(c_1,\ldots,c_m)$ from $y$ to $x$ we set
\[\p(c)=\prod_{i=1}^{m-1}\p_{c_{i+1}\to c_{i}}.\]

A \emph{weighted graded (pure) poset} (called \emph{measured poset} in \cite{DDFH}) of rank $d$ is a weighted and graded poset $P$ such that $\m(\smallest)=1.$
We write \[\m_i=\m|_{P(i)},~D_i = D|_{C^i}:C^i\to C^{i-1},~U_i=U|_{C^i}:C^i\to C^{i+1}.\]
\begin{obs}\label{obs:properties_of_weights}
Let $P$ be a (finite, pure) graded weighted poset of rank $d$.
Then for any $k\leq l\leq d,$ and $x\in P(k)$,
\[\m(x)=\sum_{y\in P(l)}\sum_{c\in \Ch(y\to x)}\p(c)\m(y).\]
Consequently, each $\m_i$ is a probability distribution.  
In addition, for all $y\in P(l),$ \begin{equation}\label{eq:conservation of probs}
\sum_{x\in P(k)}\sum_{c\in \Ch(y\to x)}\p(c)=1.
\end{equation}
\end{obs}

The weight function $\m$ induces a non degenerate inner product on $C^i$ given by
\[\forall f,g\in C^i,~\langle f,g\rangle = \sum_{x\in P_i}\m(x)f(x)g(x).\]
We write $\nr-\nr$ for the induced norm.

\begin{lemma}\label{lem:U,D dual}
Under the above assumptions, $U_i$ is dual to $D_{i+1}.$
\end{lemma}
\begin{proof}
\begin{align*}
\langle f, U_i g\rangle &= \sum_{x\in P(i+1)}\m(x)f(x)(U_ig)(x)=\\
&=\sum_{x\in P(i+1)}\m(x)f(x)\sum_{y\lhd x}\p_{x\to y}g(y)=\\
&=\sum_{y\in P(i)}\m(y)\left(\sum_{x\rhd y}\frac{\m(x)\p_{x\to y}}{\m(y)}f(x)\right)g(y)=\langle D_{i+1}f,g\rangle.
\end{align*}
\end{proof}
\begin{lemma}\label{lem:D,U and constants}
Whenever defined, $D_l\one=U_l\one=\one.$
In addition, for every $f\in C^0$
\[D_0(f)=\langle f,\one\rangle.\]
\end{lemma}
\begin{proof}
By the definition of the transition probabilities,
\[(U_l\one)(x) =\sum_{y\lhd x}\p_{x\to y}=1, \]and by \eqref{eq:weight}
\[(D_l\one)(x)=\sum_{y\rhd x}\frac{\m(y)\p_{y\to x}}{\m(x)}=1.\]
When $\smallest$ is the unique minimal element, for every $y\in P(0),~\p_{y\to\smallest}=1$ and $\m(\smallest)=1.$ Thus, \[\forall f\in C^0,~~
D_{0}f(\smallest)=\sum_{y\rhd \smallest}\m(y)f(y)=\langle f,\one\rangle.\]
\end{proof}

\subsubsection{The upper and lower walks}
Using the operators $(D_i,U_i)_i$ we can define natural random walks on $P(i).$
The $i-$th lower walk is the random walk on $C^i,~i\geq0$ induced by $M^{-}_i=U_{i-1}D_i,$ the $i-$th upper walk is the random walk on $C^i,$ for $i$ less than the rank of $P,$ is random walk induced by $M^+_i=D_{i+1}U_i.$
\begin{lemma}\label{lem:upper and lower walks}
Consider $f\in C^i.$ The upper and lower walks can be described as follows.
\begin{equation}\label{eq:upper_walk}
(M^+_i(f))(y)=
\left(\sum_{z\rhd y}\frac{\m(z)\p^2_{z\to y}}{\m(y)}\right)f(y)+
\sum_{x\in P(i),~x\neq y}\left(\sum_{z\rhd x,y}\frac{\m(z)\p_{z\to x}\p_{z\to y}}{\m(y)}\right)f(x).
\end{equation}
\begin{equation}\label{eq:lower_walk}
(M^-_i(f))(y)=\sum_{x\in P(i)}\left(\sum_{z\lhd x,y}\frac{\p_{x\to z}\p_{y\to z}}{\m(z)}\right)\m(x)f(x).
\end{equation}
\end{lemma}
\begin{proof}
\begin{align*}
(D_{i+1}U_i(f))(y)&=\sum_{z\rhd y}\frac{\m(z)\p_{z\to y}}{\m(y)}(U_i(f))(z)=\\
&=\sum_{z\rhd y}\frac{\m(z)\p_{z\to y}}{\m(y)}\sum_{x\lhd z}\p_{z\to x}f(x)=\sum_{x\in P(i)}\left(\sum_{z\rhd x,y}\frac{\m(z)\p_{z\to x}\p_{z\to y}}{\m(y)}\right)f(x).
\end{align*}
The last equality is obtained by changing the order of summation. \eqref{eq:upper_walk} is obtained from the last expression by separating into the cases $x=y$ and $x\neq y.$
Similarly,
\begin{align*}
&(U_{i-1}D_i(f))(y)=\sum_{z\lhd y}\p_{y\to z}(D_if)(z)=\\
\qquad&=\sum_{z\lhd y}\p_{y\to z}\sum_{x\rhd z}\left(\frac{\m(x)\p_{x\to z}}{\m(z)}\right)f(x)=\sum_{x\in P(i)}\left(\m(x)\sum_{z\lhd x,y}\frac{\p_{x\to z}\p_{y\to z}}{\m(z)}\right)f(x).
\end{align*}
\end{proof}
\begin{obs}\label{obs:same_spec}
$M^+_k,M^-_{k+1}$ have the same non-zero eigenvalues, including multiplicities.
\end{obs}
Indeed, this is always the case for two operators of the form $AA^*,A^*A,$ and  \[M^+_k=D_{k+1}U_k,~M^{-}_{k+1}=U_kD_{k+1},~D_{k+1}=U_k^*.\]

\begin{ex}\label{ex:3}
In case $P$ is a standard weighted graded poset, a single step of the lower random walk from $y$ is obtained by choosing a uniformly random $z$ covered by $y,$ and then choosing $x$ which covers $z,$ according to the transition probabilities $\frac{\m(x)}{\NN(x)\m(z)}.$ A single step of the upper random walk from $y$ is obtained by choosing a random $z$ which covers $y,$ with probability $\frac{\m(z)}{\NN(z)\m(y)},$ and then choosing uniformly a random element $x\lhd z.$
\end{ex}

An immediate corollary of Lemma \ref{lem:D,U and constants} is
\begin{cor}\label{cor:projection_to_consts}
If $P$ has a unique minimum $\smallest\in P(-1)$ then
\[M^+_0M^-_0=M^-_0. \]
\end{cor}
\begin{definition}\label{def:adj_matrix}
We define the \emph{$l$-th adjacency matrix} of $P$ by its action on $f\in C^l:$
\begin{equation}\label{eq:adjacency_general}
(A_l(f))(y)=
\sum_{x\neq y}\left(\sum_{z\rhd x,y}\frac{\m(z)\p_{z\to x}\p_{z\to y}}{(1-\p_{z\to y})\m(y)}\right)f(x).\end{equation}\footnote{In what follows, whenever we work with the $l-$th adjacency matrix we shall implicitly assume that any $z\in P(l+1),$ for covers at least two elements, otherwise the adjacency matrix is ill defined. This assumption is valid in almost all interesting applications.}
The zeroth adjacency matrix is sometimes referred as the adjacency matrix, and we sometimes write only $A.$
\end{definition}
$\one\in C^l$ is always a eigenvector for eigenvalue $1$ of the adjacency matrix. The eigenvalues of eigenvectors which are not constant are called \emph{non trivial eigenvalues}.
If the weight scheme is standard the adjacency operator is self-adjoint. If, in addition, each $z\in P(l+1)$ covers the same number of elements $\NN(z)=\Nlow_{l+1}$ (this property will be called 'lower regular' in what follows), then $A_l$ is related to $M^+_l$ by
\begin{equation}\label{eq:adb_vs_upper}A_l = \frac{\Nlow_{l+1}-1}{\Nlow_{l+1}}M^+_l +\frac{1}{\Nlow_{l+1}}Id_{C^l}.\end{equation}

\subsection{Links and localization}\label{subsec:links}
Let $P$ be a poset, and $x\in P.$ We will call the subposet made of all elements $y\geq x$ the \emph{link of $x$} and we shall denote it by $P_x.$
If $P$ is graded, then also $P_x,$ and the induced rank function $\rk_x$ is given by
\[\rk_x(y)=\rk(y)-\rk(x)-1.\]
If $P$ is graded, weighted (finite and pure), then we can also induce the weight function and transition probabilities by putting
\begin{equation}\label{eq:induced m}
\m_x(y)=\frac{\m(y)}{\m(x)}\sum_{c\in\Ch(y\to x)}\p(c),\end{equation}
\begin{equation}\label{eq:induced p}
(\p_x)_{z\to y}=\frac{\p_{z\to y}\sum_{c\in\Ch(y\to x)}\p(c)}{\sum_{c\in\Ch(z\to x)}\p(c)}.\end{equation}
We need to verify that the properties of transition probabilities and weights hold for the induced weights and probabilities.
We will verify Property \eqref{eq:weight}, the other properties are straight forward:
\begin{align*}\sum_{z\rhd y}\m_x(z)(\p_x)_{z\to y}
&=\sum_{z\rhd y}\frac{\m(z)}{\m(x)}\sum_{c\in\Ch(z\to x)}\p(c)
\frac{\p_{z\to y}\sum_{c\in\Ch(y\to x)}\p(c)}{\sum_{c\in\Ch(z\to x)}\p(c)}=\\
&=\sum_{z\rhd y}\frac{\m(z)\p_{z\to y}}{\m(x)}\sum_{c\in\Ch(y\to x)}\p(c)=\\&=\frac{\m(y)}{\m(x)}\sum_{c\in\Ch(y\to x)}\p(c)=\m_x(y),
\end{align*}
the one before last passage used \eqref{eq:weight} for $\m,$ and the last passage is just the definition of $\m_x.$

It is also straightforward to observe that the localization of a standard graded weighted poset is "standard in rank $1$" in the sense that-
\begin{obs}\label{obs:standard_under_loc}
If $P$ is a standard weighted graded poset of rank $d$, then for any $x\in P(\leq d-2),$ and any $z\in P_x(1),$ for any $y\in P_x(0),$ with $y\lhd z$
\[(\p_x)_{z\to y}=\frac{1}{|\{w\in P_x(0)|w\lhd z\}|}.\]
\end{obs}

We write $C^i_x=C^i(P_x)$ for the space of real functions on $P_x,$ and write $\langle-,-\rangle_x,\nr-\nr_x$ for the induced inner products and norms, we can also define the operator $D_{x,i}, U_{x,i}$ as above, only with $\m_x,\p_x$ instead of $\m,\p.$
The \emph{localization} is the linear map from $C^i$ to $C^{i-\rk(x)-1}_x$ which maps $f$ to $f_x,$ defined by $f_x(y)=f(y)$ for $y\in P_x.$

A poset $P$ is \emph{connected} if the $0-$th upper walk is irreducible, i.e. for any $y,z\in P(0),$ there exists a power $j>0$ such that the $z$ entry of $(M^+_{0})^je_y$ is non zero, where $e_y$ is the vector whose $w$ component is $\delta_{yw}.$ $P$ is \emph{locally connected} if $P$ and all of its proper links, for any $x$ of rank at most $d-2,$ are connected.

The following simple but powerful proposition generalized a classic result by Garland \cite{G} to our setting. It provides a decomposition of inner products and of inner products of the operator $D$ to sum of local terms.
\begin{prop}\label{prop:garland_posets_1}
Let $P$ be a graded weighted poset of rank $d.$ Let $-1\leq k < l\leq d,$ and $f,g\in C^l.$ Then
\begin{enumerate}
\item \begin{equation}\label{eq:localization_of_inner_prod}
\langle f,g\rangle = \sum_{x\in P(k)}\m(x)\langle f_x,g_x\rangle_x.\end{equation}
\item \begin{equation}\label{eq:localization_of_D_i}
\langle D_l f,D_l g\rangle = \sum_{x\in P(k)}\m(x)\langle D_{x,l-k-1}f_x,D_{x,l-k-1}g_x\rangle_x.\end{equation}
\end{enumerate}
\end{prop}
\begin{proof}
For \eqref{eq:localization_of_inner_prod},
\begin{align*}
\sum_{x\in P(k)}\m(x)\langle f_x,g_x\rangle_x &=
\sum_{x\in P(k)}\m(x)\sum_{y\in P_x(l-k-1)}\m_x(y)f_x(y)g_x(y)=\\
&=\sum_{x\in P(k)}\m(x)\sum_{y\in P_x(l-k-1)}\left(\frac{\m(y)}{\m(x)}\sum_{c\in \Ch(y\to x)}\p(c)\right)f(y)g(y)=\\
&=\sum_{x\in P(k)}\sum_{y\in P_x(l-k-1)}\left(\m(y)\sum_{c\in \Ch(y\to x)}\p(c)\right)f(y)g(y)=\\
&=\sum_{y\in P(l)}\m(y)f(y)g(y)\left(\sum_{x\in P(k)}\sum_{c\in \Ch(y\to x)}\p(c)\right)=\\
&=\sum_{y\in P(l)} \m(y)f(y)g(y)=\langle f,g\rangle,
\end{align*}
the first three equalities and the last one follow from the definitions, the fourth one from changing order of summation and the fifth one from \eqref{eq:conservation of probs}.

For \eqref{eq:localization_of_D_i},
\begin{align*}
\sum_{z\in P(k)}\m(z)&\langle D_{z,l-k-1}f_z,D_{z,l-k-1}g_z\rangle_z
=\\&=
\sum_{z\in P(k)}\m(z)\sum_{x\in P_z(l-k-2)}\m_z(x)\sum_{y_1,y_2\rhd x}\frac{\m_z(y_1)(\p_z)_{y_1\to x}}{\m_z(x)}f_z(y_1)\frac{\m_z(y_2)(\p_z)_{y_2\to x}}{\m_z(x)}g_z(y_2)=\\
&=\sum_{y_1,y_2\in P(l)}f(y_1)g(y_2)\sum_{x\in P({l-1}),~x\lhd y_1,y_2}\sum_{z\in P_k,~z<x}\m(z)
\frac{\m_z(y_1)(\p_z)_{y_1\to x}\m_z(y_2)(\p_z)_{y_2\to x}}{\m_z(x)}=\\
&=\sum_{y_1,y_2\in P(l)}f(y_1)g(y_2)\sum_{\substack{x\in P({l-1}),\\x\lhd y_1,y_2}}\sum_{z\in P_k,~z<x}\m(z)
\frac{\m(z)}{\m(x)\sum_{c\in \Ch(x\to z)}\p(c)}\cdot\\
&\quad\quad\cdot\frac{\m(y_1)}{\m(z)}\sum_{c\in\Ch(y_1\to z)}\p(c)\frac{\p_{y_1\to x}\sum_{c\in\Ch(x\to z)}\p(c)}{\sum_{c\in\Ch(y_1\to z)}\p(c)}
\frac{\m(y_2)}{\m(z)}\sum_{c\in\Ch(y_2\to z)}\p(c)\frac{\p_{y_2\to x}\sum_{c\in\Ch(x\to z)}\p(c)}{\sum_{c\in\Ch(y_2\to z)}\p(c)}=\\
&=\sum_{y_1,y_2\in P(l)}\m(y_1)f(y_1)\m(y_2)g(y_2)\sum_{x\in P({l-1}),~x\lhd y_1,y_2}\frac{1}{\m(x)}\p_{y_1\to x}\p_{y_2\to x}\sum_{z\in P_k,~z<x}
{\sum_{c\in\Ch(x\to z)}\p(c)}=\\
&=\sum_{y_1,y_2\in P(l)}\m(y_1)f(y_1)\m(y_2)g(y_2)\sum_{x\in P({l-1}),~x\lhd y_1,y_2}\frac{1}{\m(x)}\p_{y_1\to x}\p_{y_2\to x}=\\
&=\sum_{x\in P(l-1)}\m(x)\sum_{y_1\rhd x}\frac{\m(y_1)\p_{y_1\to x}}{\m(x)}f(y_1)\sum_{y_2\rhd x}\frac{\m(y_2)\p_{y_2\to x}}{\m(x)}g(y_2)=\langle D_l f,D_l g\rangle.
\end{align*}
the third equality uses the definitions of the localizations of the weights and transition probabilities, the second before last uses \eqref{eq:conservation of probs}.
\end{proof}

\subsection{Regularity properties and their basic consequences}\label{subsec:axioms}
In practice, many posets which appear in the literature and in applications have more structure. In Subsection \ref{subsec:regularity_intro} we described several \emph{structural regularity properties} possessed by many posets of interest. We now study the most basic properties of such posets. We will see throughout this article that these additional structural properties allow generalizing many non trivial results from the class of simplicial complexes to more general graded weighted posets.
More precisely, in later sections we will generalize these structural properties to properties of weighted graded posets.
We will show that the more general properties yield strong spectral consequences. Then we will see that \emph{standard} weighted graded posets which possess some of the structural regularity, also possess corresponding more general properties of the weight schemes, hence their consequences.

Recall the regularity properties of Subsection \ref{subsec:regularity_intro}. When $P$ is lower regular, middle and $\wedgeV$ regular, the corresponding structure constants are not independent. Indeed, for a given $z\in P(l+1),$
the number of triples $(x,y,s)$ with $x\neq y,$ $z\rhd x,y,$ and $x,y\rhd s$ is, on the one hand  $\Nlow_{l+1}(\Nlow_{l+1}-1)\NwedgeV_l,$ if we first count $x,y$ and then $s.$ But if we first count $x,$ then $s$ and then $y$ we obtain $\Nlow_{l+1}\Nlow_l(\Nmid_l-1).$ Thus
\begin{equation}\label{eq:Ns_relation}\frac{\Nlow_l(\Nmid_l-1)}{(\Nlow_{l+1}-1)\NwedgeV_l}=1.\end{equation}
\begin{ex}\label{ex:structural_axioms_for_simplex_and_grass}
If $P$ is a simplicial complex then it
is lower regular with $\Nlow_i=i+1,$ it is middle regular with $\Nmid_i=2$ and it is $\wedgeV$ regular with $\NwedgeV=1.$
If $P$ is a Grassmannian poset, where the ground field $\F_q$ has $q$ elements, then it
is lower regular with $\Nlow_i=[i+1]_q,$ it is middle regular with $\Nmid_i=q+1$ and it is $\wedgeV$ regular with $\NwedgeV=1.$
\end{ex}

Note that for a poset $P$ which is lower regular at levels $1,2$ and middle regular at level $1,$ for every $u\in P(2),$
\begin{equation}\label{eq:2_desc}|\{x\in P(0)|x<u\}|=\frac{\Nlow_2\Nlow_1}{\Nmid_1}.\end{equation}
Indeed, there are $\Nlow_2\Nlow_1$ chains of length $2$ descending from $u$, and by the definition of middle regularity, they are grouped into groups of size $\Nmid_1$ of chains having the same endpoints. Thus, the number of different descendants of $u$ is $\frac{\Nlow_2\Nlow_1}{\Nmid_1}.$

If $P$ is $2-$skeleton regular then
\begin{equation}\label{eq:2-skel_rel}
\Rwye = \frac{\Nmid_1(\Nlow_1-1)}{\frac{\Nlow_2\Nlow_1}{\Nmid_1}-1}.\end{equation}
Indeed, given $P(2)\ni u>y\in P(0)$ we count pairs of the form
\[\{(x,z)\in P(0)\times P(1)|x\neq y, u\rhd z\rhd x,y\}.\]
On the one hand, if we first count $x,$ and then $z,$ we obtain, by \eqref{eq:2_desc} $\frac{\Nlow_2\Nlow_1}{\Nmid_1}-1$ choices for $x,$ and then, by definition $\Rwye$ choices of $z.$ All together this number is $(\frac{\Nlow_2\Nlow_1}{\Nmid_1}-1)\Rwye.$
On the other hand, we can first choose $z,$ there are $\Nmid_1$ ways to do that. Then, there are $\Nlow_1-1$ choices for an element $x\neq y$ covered by $z.$ The total number is then $\Nmid_1(\Nlow_1-1).$ Comparing the expressions, \eqref{eq:2-skel_rel} follows.

For any poset property $Q,$ e.g. regularity, lower regularity at level $1,$ etc., we say that \emph{$P$ is locally $Q$} if $P,$ and all its links on which it is possible to verify property $Q$ possess property $Q,$ and if $Q$ depends on parameters, e.g. $\Nlow_i,$ then the parameters of $Q$ at the link depend only on $\rho(s).$
For example, being middle regular is equivalent to being locally level $1$ lower regular.

Another example is that a graded poset $P$ of rank $d$ is \emph{locally $\wye$ regular} if there exist constants $\Rwye_i,~-1\leq i\leq d-3,$ such $P$ and all its links $P_s,~s\in P(\leq d-3)$ are $\wye$ regular, and moreover, the constants depend only on the level, in the sense that $\Rwye(P_s)=\Rwye_{\rho(s)}.$

A graded poset $P$ is $2-$skeleton regular if it is lower regular at levels $1,2,$ middle regular at level $1,$ 
and $\wye$ regular. Similarly it is locally $2-$skeleton regular if for all $s\in P(\leq d-3),$ $P_s$ is $2-$skeleton regular with uniform structure constants. The lower, middle and $\wedgeV$ regularity constants at a link of level $i$ will be denoted $\Nlow_{i,1},~\Nlow_{i,2},~\Nmid_{i,1},
$ respectively.
\begin{ex}
Simplicial complexes and Grassmannian posets are locally $\wye$ regular with $\Rwye_i=1$ for all $i.$ They are also locally $2-$skeleton regular, where for both types of posets \[\Nlow_{i,1}=\Nlow_1,~\Nlow_{i,2}=\Nlow_2,~\Nmid_{i,1}=\Nmid_1,~\forall i.\]
\end{ex}

\subsection{Expansion notions for posets}\label{subsec:expandefs}
For a diagonalizable $n\times n$ matrix $M$ with eigenvalues $\lambda_1\geq\lambda_2\geq\ldots\geq\lambda_n,$ we write $\lambda(M)=\max\{\lambda_2,|\lambda_n|\}.$


\bigskip\noindent
We shall now provide local and global definitions of expansion in posets.

\bigskip\noindent
We first define local expansion notions. 
\begin{definition}
Let $P$ be a standard\footnote{In this definition and the next one we require $P$ to be standard so that the adjacency matrices will be self adjoint, and will thus be diagonalizable with real eigenvalues.}, weighted, graded poset of rank $d$. $P$ is called \emph{one-sided $\lambda$-local spectral expander} if $P$ and any link $P_x,$ for $x\in P(\leq d-2)$ are connected, and the non trivial eigenvalues of the adjacency matrix of the link $P_x$, for every $x \in P(\leq d-2),$ are bounded by $\lambda.$
\end{definition}

\begin{definition}
Let $P$ be a standard, weighted, graded poset of rank $d$. $P$ is called a \emph{two-sided $[\nu,\lambda]$-local spectral expander} if $P$ and any link $P_x,$ for $x\in P(\leq d-2)$ are connected, and the non trivial eigenvalues of the adjacency matrix of the link $P_x$, for every $x \in P(\leq d-2)$ lie in $[\nu,\lambda].$
\end{definition}

In \cite{DDFH} the following definition of expansion was given, and they termed posets which satisfy this definition \emph{eposet}. We shall call these posets \emph{global eposets}, to distinguish them from the local ones defined above, and because the definition is via a global criterion.\footnote{Note that this definition does not assume that the adjacency matrix of the links have real eigenvalues hence it applies also to non standard posets.}

\begin{definition}\label{def:eposet}
A weighted graded poset $P$ of rank $d$ is a \emph{$\lambda$-global eposet} if for all $1 \leq j\leq d-1$ there exist constants $r_j,\delta_j$ such that
\[\nr D_{j+1}U_j-\delta_jU_{j-1}D_j - r_j Id_{C^j} \nr\leq\lambda.\]
\end{definition}


\section{The Up Localization Property and its applications}\label{sec:Garland}
In this section we describe property UL\footnote{
'UL' is for 'Up-Localization' a name that will be justified in Proposition \ref{prop:garland_posets_2}.
} and its consequences. We shall see that a poset which has the UL property allows a localization decomposition of $\langle Uf, Ug\rangle,$ and, if it is a one sided local spectral expanding poset, then its function spaces $C^i$ have a decomposition into approximate eigenspaces, generalizing a result of \cite{KO}. We will show that any standard, regular (meaning lower, middle and $\wedgeV$ regular) poset also possesses the UL property.

In Section \ref{sec:2sided} we will consider a closely related property, the Adjacency Localization (AL) Property, which also generalizes regularity. 
Some of the consequences that we present for the two properties are similar, but since there are weighted posets which are better behaved with respect to the UL property, and some which are better behaved with respect to the AL property, we chose not to neglect similar results. 
Moreover, the AL property requires the poset to be standard, a property we do not require in this section. 
\subsection{The Up Localization property}
\begin{definition}\label{def:ass_I}
Let $(P,\leq,\rk,\m,\p)$ be a weighted graded poset of rank $d.$
\begin{enumerate}
\item\textbf{UL.1 }
We say that $P$ possesses Property UL.1 if there exist constants $\cxyz_0,\ldots,\cxyz_{d-1}$ such that for all $y\in P(l),~0\leq l<d,$ and any $x$ which covers $y,$
\[\sum_{z\in P(l-1),~z~\lhd y}\frac{\p^2_{y\to z}}{\sum_{x\in\Ch(x\to z)}\p(c)}=\cxyz_l.\]
\item\textbf{UL.2 }
We say that $P$ possesses Property UL.2 if there exist constants $\cdia_0,\ldots,\cdia_{d-1}$ such that for all $y_1\neq y_2\in P(l),~0\leq l<d,$ which are both covered by at least one element, and any $x$ which covers both,
\[\sum_{z\in P(l-1),~z\lhd y_1,y_2}\frac{\p_{y_1\to z}\p_{y_2\to z}}{\sum_{x\in\Ch(x\to z)}\p(c)}=\cdia_l.\]
\item\textbf{UL.3 }
We say that $P$ possesses Property UL.3 if there exist constants $\csqr_0,\ldots,\csqr_{d-1}$ such that for all $y\in P(l),~0\leq l<d,$
\[\sum_{x\rhd y}\m(x)\p^2_{x\to y}=\csqr_l\m(y).\]
\end{enumerate}
We will say that $P$ possesses Property UL if it possesses Properties UL.1-UL.3.
\end{definition}

The next lemma, whose proof is direct, relates structural properties with weight properties.
\begin{lemma}\label{lem:special_cases_for_xyz_and_dia}
If the \emph{standard} weighted graded poset $P$ is lower regular and middle regular with constants $(\Nlow_i)_i, (\Nmid_i)_i$ respectively, then it possesses properties UL.1, UL.3 with constants \[\cxyz_i=\Nlow_i\frac{(\Nlow_i)^{-2}}{\Nmid_i(\Nlow_{i+1}\Nlow_i)^{-1}}
=\frac{\Nlow_{i+1}}{\Nmid_i},~~~\csqr_i=\frac{1}{\Nlow_{i+1}}.\]
If $P$ is in addition $\wedgeV$ regular, then the standard weighted poset will also possesses property UL.2 with constants
\[\cdia_i=\NwedgeV_i\frac{(\Nlow_i)^{-2}}{\Nmid_i(\Nlow_{i+1}\Nlow_i)^{-1}}=
\frac{\NwedgeV_i\Nlow_{i+1}}{\Nlow_i\Nmid_i}.\]
\end{lemma}
\begin{ex}\label{ex:axioms_for_simplex_and_grass}
Using Example \ref{ex:structural_axioms_for_simplex_and_grass} and Lemma \ref{lem:special_cases_for_xyz_and_dia} we see that if
$P$ is a simplicial complex with an arbitrary standard weight scheme then it possesses property UL with constants \[\cxyz_l=\frac{l+2}{2},~~\cdia_l=\frac{l+2}{2(l+1)},~~ \csqr_l=\frac{1}{l+2}.\]
Similarly, if $P$ is a Grassmannian poset, where the ground field $\F_q$ has $q$ elements, with an arbitrary standard weight scheme then it possesses property UL with constants \[\cxyz_l=\frac{[l+2]_q}{q+1},
~~\cdia_l=
\frac{[l+2]_q}{(q+1)[l+1]_q},~~\csqr_l=\frac{1}{[l+2]_q}.\]
\end{ex}

The following observation is immediate from Lemma \ref{lem:upper and lower walks}
\begin{obs}\label{obs:upper_walk_special_case}
If $P$ possesses property UL.3 then
\begin{equation*}
(D_{l+1}U_l(f))(y)=
\csqr_lf(y)+
\sum_{x\neq y}\left(\sum_{z\rhd x,y}\frac{\m(z)\p_{z\to x}\p_{z\to y}}{\m(y)}\right)f(x).
\end{equation*}
\end{obs}

The usefulness of Property UL is seen in the following proposition, which generalizes Garland's argument \cite{G}, this time for the inner product of the operators $U_i.$
\begin{prop}\label{prop:garland_posets_2}
Let $P$ be a graded weighted poset of rank $d,$ which possesses Property UL with constants $(\cxyz_i)_i,(\csqr_i)_i,(\cdia_i)_i$ respectively. Then for any $0\leq l\leq d-1$ and $f,g\in C^l$
\[\langle U_lf,U_lg\rangle=
\frac{1}{\cdia_l}\sum_{z\in P(l-1)}\m(z)\langle U_{z,0}(f_z),U_{z,0}(g_z)\rangle_z-
\csqr_l(\frac{\cxyz_l}{\cdia_l}-1)\langle f,g\rangle.\]
\end{prop}
\begin{proof}
We expand the terms in the equation above. First,
\begin{align}\label{eq:U_iU_i1}
\langle U_lf,U_lg\rangle &=\sum_{x\in P(l+1)}\m(x)(U_if)(x)(U_ig)(x)=\\
\notag&=\sum_{x}\m(x)\sum_{y_1,y_2\lhd x}\p_{x\to y_1}f(y_1)\p_{x\to y_2}f(y_2)=\sum_{y_1,y_2\in P(l)}f(y_1)f(y_2)\sum_{x\rhd y_1,y_2}\m(x)\p_{x\to y_1}\p_{x\to y_2}.
\end{align}
Second,
\begin{align}\label{eq:U_iU_i2}
\sum_{z\in P(l-1)}&\m(z)\langle U_{z,0}(f_z),U_{z,0}0(g_z)\rangle_z=
\sum_{z\in P(l-1)}\m(z)\sum_{x\in P_z(1)}\m_z(x)(U_{z,0}f_z)(x),(U_{z,0}g_z)(x)=\\
\notag&=\sum_{z\in P(l-1)}\m(z)\sum_{x\in P_z(1)}\m_z(x)\sum_{y_1,y_2\in P_z(0),y_1,y_2\lhd x}(\p_z)_{x\to y_1}f(y_1)(\p_z)_{x\to y_2}g(y_2)=\\
\notag&=\sum_{z\in P(l-1)}\m(z)\sum_{x\in P_z(1)}\frac{\m(z)\sum_{c\in\Ch(x\to z)}\p(c)}{\m(x)}
\sum_{y_1,y_2\lhd x}f(y_1)g(y_2)
\frac{\p_{x\to y_1}\p_{y_1\to z}}{\sum_{c\in \Ch(x\to z)}\p(c)}\frac{\p_{x\to y_2}\p_{y_2\to z}}{\sum_{c\in \Ch(x\to z)}\p(c)}=\\
\notag&=\sum_{y_1,y_2\in P(l)}f(y_1)g(y_2)\sum_{x\rhd y_1,y_2}\m(x)\p_{x\to y_1}\p_{x\to y_2}\sum_{z\lhd y_1,y_2}\frac{\p_{y_1\to z}\p_{y_2\to z}}{\sum_{c\in\Ch(x\to z)}\p(c)}=\\
\notag&=\cdia_l\sum_{y_1\neq y_2\in P(l)}f(y_1)g(y_2)\sum_{x\rhd y_1,y_2}\m(x)\p_{x\to y_1}\p_{x\to y_2}
+\cxyz_l\sum_{y\in P(l)}f(y)g(y)\sum_{x\rhd y}\m(x)\p^2_{x\to y}=\\
\notag&=\cdia_l\sum_{y_1,y_2\in P(l)}f(y_1)g(y_2)\sum_{x\rhd y_1,y_2}\m(x)\p_{x\to y_1}\p_{x\to y_2}
+(\cxyz_l-\cdia_l)\sum_{y\in P(l)}f(y)g(y)\sum_{x\rhd y}\m(x)\p^2_{x\to y}=\\
\notag&=\cdia_l\sum_{y_1,y_2\in P(l)}f(y_1)g(y_2)\sum_{x\rhd y_1,y_2}\m(x)\p_{x\to y_1}\p_{x\to y_2}+\csqr_l(\cxyz_l-\cdia_l)\sum_{y\in P(l)}f(y)g(y)=\\
\notag&=\cdia_l\langle U_lf,U_lg\rangle+\csqr_l(\cxyz_l-\cdia_l)\langle f,g\rangle,
\end{align}
the first passages involved definitions, the fourth one was changing the order of summation, the third to last passage used properties $\xyz$ and $\dia,$ the one before last passage used property $\sqr$, the last passage used \eqref{eq:U_iU_i1} and the definition of the inner product.
The claim now follows.
\end{proof}
Under the UL assumption, the constants $\cdia_l,\cxyz_l$ and $\csqr_l$ are constrained.
\begin{lemma}\label{lem:ineq_dia_xyz}
If $P$ possesses properties UL.1,UL.2 then the corresponding constants satisfy,
\[\cdia_l\leq\cxyz_l,\]for all $l.$
If $P$ possesses property UL then $\cdia_l<1,$ and
\begin{equation}\label{eq:c_relation}
\frac{1}{\cdia_l}-\csqr_l(\frac{\cxyz_l}{\cdia_l}-1)=1.
\end{equation}
\end{lemma}
\begin{proof}
For the first claim, consider $y_1\neq y_2\in P(l)$ which are covered by a common $x.$ If there are no such then $\cdia_l=0\leq\cxyz_l.$
\begin{align}\label{eq:ineq_dia_xyz_not_nec_const}
\cdia_l&=\sum_{z\in P(l-1),~z~\lhd y_1,y_2}\frac{\p_{y_1\to z}\p_{y_2\to z}}{\sum_{c\in\Ch(x\to z)}\p(c)}\leq\\
\notag&\leq \frac{1}{2}\sum_{z\in P(l-1),~z\lhd y_1,y_2}\frac{\p^2_{y_1\to z}+\p^2_{y_2\to z}}{\sum_{c\in\Ch(x\to z)}\p(c)}\leq\\
\notag&\leq \frac{1}{2}\left(\sum_{z\in P(l-1),~z\lhd y_1}\frac{\p^2_{y_1\to z}}{\sum_{c\in\Ch(x\to z)}\p(c)}+\sum_{z\in P(l-1),~z\lhd y_2}\frac{\p^2_{y_2\to z}}{\sum_{c\in\Ch(x\to z)}\p(c)}\right)=\cxyz_l.
\end{align}
The second claim follows from substituting $f=g=\one$ in Proposition \ref{prop:garland_posets_2}. Then the LHS becomes $\nr\one\nr^2=1,$ and the RHS becomes
\[\frac{1}{\cdia_l}\sum_{z\in P(l-1)}\m(z)\nr U_{z,0}\one_z\nr^2-\csqr_l(\frac{\cxyz_l}{\cdia_l}-1)\nr\one\nr^2=\frac{1}{\cdia_l}\sum_{z\in P(l-1)}\m(z)-\csqr_l(\frac{\cxyz_l}{\cdia_l}-1)=
\frac{1}{\cdia_l}-\csqr_l(\frac{\cxyz_l}{\cdia_l}-1).\]
Since the constants are non negative and,  $\frac{\cxyz_l}{\cdia_l}-1\geq0,$ the second claim shows $\frac{1}{\cdia_l}\geq 1.$
\end{proof}

\subsection{A decomposition theorem for one-sided local spectral expanding posets}
In this section we use the previous results to generalize spectral decomposition theorems for $1-$sided local high dimensional expanders. We first extend results of the first author and Oppenheim \cite[Theorems 1.3, 1.4]{KO}, to the setting of more general posets with an upper bound on the second eigenvalue of the adjacency matrices of the links. Our proof generalizes their arguments, and the novelty is in identifying the exact assumptions required for this generalization to hold. 
\begin{prop}\label{prop:towards UD-DU}
If $P$ possesses Property UL, then $\forall\alpha,\beta\in\R$
\[\cdia_l\langle U_l f, U_lg\rangle=
\beta(1-\alpha)\langle D_l f, D_lg\rangle +(\alpha -\csqr_l(\cxyz_l-\cdia_l))\langle f,g\rangle
+\sum_{x\in P(l-1)}\m(x)\langle f_x,(M^+_{x,0}-\alpha Id)(Id-\beta M^-_{x,0})g_x\rangle_x.\]
\end{prop}
\begin{proof}
$\langle U_{x,0}f_x,U_{x,0}g_x\rangle_x=\langle f_x, M_{x,0}^+g_x\rangle_x=\langle f_x, (M^+_{x,0}-\alpha Id)g_x\rangle_x+\alpha\langle f_x,g_x\rangle_x=$
\begin{align*}
&=\langle f_x, (M^+_{x,0}-\alpha Id)g_x\rangle_x+\alpha\langle f_x,g_x\rangle_x=\\&=\langle f_x, (M^+_{x,0}-\alpha Id)\beta M^-_{x,0}g_x\rangle_x+\langle f_x, (M^+_{x,0}-\alpha Id)(Id-\beta M^-_{x,0})g_x\rangle_x+\alpha\langle f_x,g_x\rangle_x=\\
&=\beta(1-\alpha)\langle f_x, M^-_{x,0}g\rangle_x+\langle f_x, (M^+_{x,0}-\alpha Id)(Id-\beta M^-_{x,0})g_x\rangle_x+\alpha\langle f_x,g_x\rangle_x=\\
&=\beta(1-\alpha)\langle D_{x,0} f_x, D_{x,0}g_x\rangle_x+\langle f_x, (M^+_{x,0}-\alpha Id)(Id-\beta M^-_{x,0})g_x\rangle_x+\alpha\langle f_x,g_x\rangle_x.
\end{align*}
the fourth equality uses Corollary \ref{cor:projection_to_consts}.
Summing over $x\in P(l-1),$ with weight $\m(x),$ and using Propositions \ref{prop:garland_posets_1},~\ref{prop:garland_posets_2} we obtain
\begin{align*}\cdia_l\langle U_l f U_lg\rangle &+\csqr_l(\cxyz_l-\cdia_l)\langle f,g\rangle=\\
&=\beta(1-\alpha)\langle D_l f, D_lg\rangle +\alpha\langle f,g\rangle
+\sum_{x\in P(k-1)}\m(x)\langle f_x,(M^+_{x,0}-\alpha Id)(Id-\beta M^-_{x,0})g_x\rangle_x,\end{align*}
\end{proof}
$M^\pm_{x,0}$ are self dual, hence they have a basis of orthonormal eigenvectors with real eigenvalues. By Perron-Frobenius the maximal eigenvalue is $1.$ In addition, $M^\pm_{x,0}$ are non negative operators:\[\langle f, M^+_{x,0}f\rangle=\langle U_{x,0}f,U_{x,0}f\rangle\geq 0,~~\langle f, M^-_{x,0}f\rangle=\langle D_{x,0}f,D_{x,0}f\rangle\geq 0.\] Write $\mu'_x,\nu'_x\in[0,1]$ for the second largest and minimal eigenvalues of $M^+_x$, repsectively. Set \begin{equation}\label{eq:links_eigenbounds}\mu'_k=\max_{x\in P(k)}\mu'_x,~\nu'_k=\min_{x\in P(k)}\nu'_x.\end{equation}
Note that if the link $P_x$ is connected, that is $M^+_{x,0}$ represents an \emph{irreducible Markov chain} then $\mu'_x<1.$
We now analyze the correction in the above proposition, in the case $\beta=1.$
\begin{prop}\label{prop:bounding_correction}For any graded weighted poset $P,~\forall f,g\in C^l$
\begin{equation}\sum_{x\in P(l-1)}\m(x)\langle (M^+_{x,0}-\alpha Id)(Id- M^-_{x,0})f_x,f_x\rangle_x\leq (\mu'_{l-1}-\alpha)_+\nr f\nr^2,\end{equation}\footnote{We write $a_+=\max\{a,0\},a_-=\min\{a,0\}.$}and
\begin{equation}\sum_{x\in P(l-1)}\m(x)|\langle (M^+_{x,0}-\alpha Id)(Id- M^-_{x,0})f_x,g_x\rangle_x|\leq \max\{\mu'_{l-1}-\alpha,\alpha-\nu'_{l-1}\}\nr f\nr\nr g\nr.\end{equation}
\end{prop}
\begin{proof}
Take $f,g\in C^l.$
\begin{align*}\langle (M^+_{x,0}-\alpha Id)(Id- M^-_{x,0})f_x,f_x\rangle_x=&
\langle(M^+_{x,0}-\alpha Id)(Id- M^-_{x,0})f_x,M^-_{x,0}f_x\rangle_x+\\
&+\langle(M^+_{x,0}-\alpha Id)(Id- M^-_{x,0})f_x,(Id-M^-_{x,0})f_x\rangle_x.\end{align*}
From Corollary \ref{cor:projection_to_consts},
\[M^-_{x,0}(Id-M^-_{x,0})=0,\] hence, from the self duality of $M^-_{x,0},Id$ the images of $M^-_{x,0},Id-M^-_{x,0}$ are orthogonal. By Corollary \ref{cor:projection_to_consts} again, also the image of $M^+_{x,0}(Id-M^-_{x,0})$ is perpendicular to that of $M^-_{x,0}.$
Plugging we get
\[\langle (M^+_{x,0}-\alpha Id)(Id- M^-_{x,0})f_x,f_x\rangle_x=
\langle(M^+_{x,0}-\alpha Id)(Id- M^-_{x,0})f_x,(Id-M^-_{x,0})f_x\rangle_x.\]
Using the definition of $\mu'_{l-1}$ we obtain
\[\langle (M^+_{x,0}-\alpha Id)(Id- M^-_{x,0})f_x,f_x\rangle_x\leq (\mu'_{l-1}-\alpha)\nr(Id-M^-_{x,0})f_x\nr_x^2\leq (\mu'_{l-1}-\alpha)_+\nr f_x\nr_x^2,\]
where the last inequality used that $f_x=M^-_xf_x+(Id-M^-_x)f_x$ is an orthogonal decomposition. Summing over $x$ and using Proposition \ref{prop:garland_posets_1} we obtain the first statement of the proposition.

For the second inequality
\begin{align*}
\sum_{x\in P(l-1)}&\m(x)|\langle (M^+_{x,0}-\alpha Id)(Id- M^-_{x,0})f_x,g_x\rangle_x|\leq\\
&\leq \sum_{x\in P(l-1)}\m(x)\nr (M^+_{x,0}-\alpha Id)(Id- M^-_{x,0})f_x\nr\nr g_x\nr\leq\\
&\leq \max\{\mu'_{l-1}-\alpha,\alpha-\nu'_{l-1}\}\sum_{x\in P(l-1)}\m(x)\nr (Id- M^-_{x,0})f_x\nr\nr g_x\nr\leq\\
&\leq\max\{\mu'_{l-1}-\alpha,\alpha-\nu'_{l-1}\}\sum_{x\in P(l-1)}\m(x)\nr f_x\nr\nr g_x\nr\leq \\ &\leq\max\{\mu'_{l-1}-\alpha,\alpha-\nu'_{l-1}\}\left(\sum_{x\in P(l-1)}\m(x)\nr f_x\nr^2\right)^{\frac{1}{2}}\left(\sum_{x\in P(l-1)}\m(x)\nr g_x\nr^2\right)^{\frac{1}{2}}\leq\\
&\leq \max\{\mu'_{l-1}-\alpha,\alpha-\nu'_{l-1}\}\nr f\nr\nr g\nr
\end{align*}
where we used Cauchy-Schwarz for the first and one before last inequalities, \ref{lem:D,U and constants} for the third inequality ($Id-M_{x,0}^-$ being an orthogonal projection), and Proposition \ref{prop:garland_posets_1} for the last one.
\end{proof}

For a graded weighted poset $P,$ write \[C_0^k=C_0^k(P)=\{f\in C^k(P)|\langle f,\one\rangle=0\}.\]
By Lemmas \ref{lem:U,D dual} and \ref{lem:D,U and constants}, whenever defined,
\begin{equation}\label{eq:5_1_first}f\in C_0^k\Leftrightarrow D_kf\in C_0^{k-1}\Leftrightarrow U_kf\in C_0^{k+1}.\end{equation}
Indeed, for $f\in C_0^k,k\geq 0,$ \[\langle D_k f,\one\rangle=0\Leftrightarrow\langle f,U_{k-1}\one\rangle=0\Leftrightarrow\langle f,\one\rangle=0,\]
and the same derivation works for $U_k,$ for $k$ smaller than the rank of $P.$
Hence, in particular
\begin{equation}\label{eq:5_1_second}M^\pm_k(C^k_0)\subseteq C^k_0~~\text{and }\text{ker} (D_k),~\text{ker} (U_k)\subseteq C^k_0.\end{equation}

\begin{prop}\label{prop:KO_5_2}
Let $P$ be a graded weighted poset of rank $d$ possessing Property UL. Let $\{\alpha_j\}_{0\leq j \leq d-1}$ be real numbers. 
Then for every $0\leq k\leq d-1,~f\in C^k_0$ there exist $g_j,h_j\in C_0^j,~0\leq j\leq k$ such that
\begin{enumerate}
\item $f=g_k.$
\item For all $j\leq k$\begin{equation}\label{eq:norm_decomp}\nr g_j\nr^2=\sum_{i=0}^j\nr h_i\nr^2.\end{equation}
\item For all $j\leq k$\begin{equation}\label{eq:norm_d_k_decomp}\nr U_j g_j\nr^2=\sum_{i=0}^ja_{j,i}\nr h_i\nr^2+\sum_{i=0}^jb_{j,i}\sum_{x\in P(i-1)}\m(x)\langle (g_i)_x,(M^+_{x,0}-\alpha_i Id)(Id- M^-_{x,0})(g_i)_x\rangle_x,\end{equation}
    where $a_{l,r}=a_{l,r}(\alpha_r,\alpha_{r+1},\ldots,\alpha_l)$ is given by
    \begin{equation}\label{eq:a_j_i}
    a_{l,r}=1-\prod_{j=r}^l\frac{1-\alpha_j}{\cdia_j},
    \end{equation} and $b_{l,r}=b_{l,r}(\alpha_r,\alpha_{r+1},\ldots,\alpha_l)$ is given by \begin{equation}b_{l,r}=\frac{1}{\cdia_r}\prod_{j=r+1}^l\frac{1-\alpha_j}{\cdia_j}\end{equation}(the empty product is taken to be $1$).
\end{enumerate}
\end{prop}
\begin{proof}
First observe that $a_{j,i},$ $b_{j,i}$ can also be defined recursively as follows \begin{equation}\label{eq:a_j_i_rec}
    a_{j,j}=\frac{\alpha_j}{\cdia_j}-\csqr_j(\frac{\cxyz_j}{\cdia_j}-1),~~a_{j,i}=\frac{1-\alpha_j}{\cdia_j}a_{j-1,i}+
    (\frac{\alpha_j}{\cdia_j}-\csqr_j(\frac{\cxyz_j}{\cdia_j}-1)),~~j>i,
    \end{equation}  \begin{equation}\label{eq:b_i_j_rec}b_{j,j}=\frac{1}{\cdia_j},~~b_{j,i}=\frac{1-\alpha_j}{\cdia_j}b_{j-1,i},~~j>i.\end{equation}
The proof for $b_{j,i}$ is immediate, while for $a_{j,i}$ it follows from Lemma \ref{lem:ineq_dia_xyz} and simple induction.

We prove the proposition by induction on $k.$
For $k=0$ we put $h_0=f$ and the proposition follows from Proposition \ref{prop:towards UD-DU} with $\beta=1,$ and Lemma \ref{lem:D,U and constants}, since $D_0f=\langle f,\one\rangle=0.$

Suppose the claim holds for $k-1.$ We have an orthogonal decomposition
\[f=h_k+g',~\nr f\nr^2=\nr h_k\nr^2+\nr g'\nr^2,~h_k\in \text{ker}(D_k)\subseteq C_0^k,~g'\in\text{ker}(D_k)^\perp\cap C^k_0=\text{Im}(U_{k-1})\cap C^k_0,\]
where the rightmost equality of spaces holds because the inner product is non degenerate.
Let $g''\in C^{k-1}$ satisfy $U_{k-1}g''=g',$ by \eqref{eq:5_1_first} $g''\in C_0^{k-1}.$
\begin{equation}\label{eq:5_3_first}
\nr f\nr^2 = \nr h_k\nr^2 + \nr U_{k-1}g''\nr^2
\end{equation}
Using Proposition \ref{prop:towards UD-DU}, with $\beta=1,$ again
\begin{align}\label{eq:5_3_second}
\nr U_kf\nr^2&=\frac{1-\alpha_k}{\cdia_k}\nr D_k f \nr^2+(\frac{\alpha_k}{\cdia_k}-\csqr_k(\frac{\cxyz_k}{\cdia_k}-1))\nr f\nr^2\\\notag&\quad\quad\quad\quad\quad\quad\quad\quad\quad\quad\quad\quad\quad+
\frac{1}{\cdia_k}\sum_{x\in P(k-1)}\m(x)\langle f_x,(M^+_{x,0}-\alpha_k Id)(Id- M^-_{x,0})f_x\rangle_x\\
\notag&=\frac{1-\alpha_k}{\cdia_k}\nr D_k g'\nr^2+(\frac{\alpha_k}{\cdia_k}-\csqr_k(\frac{\cxyz_k}{\cdia_k}-1))(\nr h_k\nr^2+\nr g'\nr^2)\\
&\notag\quad\quad\quad\quad\quad\quad\quad\quad\quad\quad\quad\quad\quad+\frac{1}{\cdia_k}\sum_{x\in P(k-1)}\m(x)\langle f_x,(M^+_{x,0}-\alpha_k Id)(Id- M^-_{x,0})f_x\rangle_x\\
\notag&=\frac{1-\alpha_k}{\cdia_k}\nr D_k U_{k-1}g''\nr^2+(\frac{\alpha_k}{\cdia_k}-\csqr_k(\frac{\cxyz_k}{\cdia_k}-1))(\nr h_k\nr^2+\nr U_{k-1}g''\nr^2)\\
&\notag\quad\quad\quad\quad\quad\quad\quad\quad\quad\quad\quad\quad\quad+\frac{1}{\cdia_k}\sum_{x\in P(k-1)}\m(x)\langle f_x,(M^+_{x,0}-\alpha_k Id)(Id- M^-_{x,0})f_x\rangle_x.
\end{align}
The operator $M^+_{k-1}=D_kU_{k-1}$ is self adjoint and non negative, hence there exists a unique square root operator $\sqrt{M^+_{k-1}}$ which is also self adjoint, non negative and satisfies $\sqrt{M^+_{k-1}}^2=M^+_{k-1}.$
Write $g_{k-1}=\sqrt{M^+_{k-1}}g'',$ and observe that
\begin{itemize}
\item $g_{k-1}\in C^{k-1}_0,$ from the same reasoning of \eqref{eq:5_1_first} as $\one$ is also an eigenfunction for $\sqrt{M^{+}_{k-1}}.$
\item $\nr g_{k-1}\nr^2=\langle \sqrt{M^+_{k-1}}g'',\sqrt{M^+_{k-1}}g''\rangle=\langle D_kU_{k-1}g'',g''\rangle=\langle U_{k-1}g'',U_{k-1}g''\rangle=\nr g'\nr^2.$
\item Similarly, \[\nr D_k U_{k-1}g''\nr^2=\nr\sqrt{M^+_{k-1}}g_{k-1}\nr^2=\nr U_{k-1}g_{k-1}\nr^2.\]
\end{itemize}
Plugging in \eqref{eq:5_3_first},~\eqref{eq:5_3_second} we obtain
\[\nr f\nr^2 = \nr h_k\nr^2+\nr g_{k-1}\nr^2,\]
\begin{align*}\nr U_kf\nr^2 &= \frac{1-\alpha_k}{\cdia_k}\nr U_{k-1}g_{k-1}\nr^2+(\frac{\alpha_k}{\cdia_k}-\csqr_k(\frac{\cxyz_k}{\cdia_k}-1))(\nr h_k\nr^2+\nr g_{k-1}\nr^2)+\\
&\quad\quad\quad\quad\quad\quad\quad\quad\quad\quad\quad\quad\quad+\frac{1}{\cdia_k}\sum_{x\in P(k-1)}\m(x)\langle f_x,(M^+_{x,0}-\alpha_k Id)(Id- M^-_{x,0})f_x\rangle_x.\end{align*}
The induction assumption, applied to $g_{k-1}$ provides $g_j,h_j\in C_0^j$ satisfying \eqref{eq:norm_decomp},~\eqref{eq:norm_d_k_decomp}. Substituting in the last two equations, and using the recursions \eqref{eq:a_j_i_rec}, \eqref{eq:b_i_j_rec}, we obtain the claim.
\end{proof}
\begin{rmk}\label{rmk:a_b_consts_nice_cases}
Suppose that $P$ is standard and regular. One natural choice for $\alpha_k$ is $\alpha_k=\frac{1}{\Nmid_k}$ since then all diagonal entries of $M^+_{x,0}-\alpha_kId$ are $0.$
By Lemma \ref{lem:special_cases_for_xyz_and_dia}, we obtain
\begin{equation}\label{eq:a_b_lazy_choice}
a_{l,r}=1-\frac{\Nlow_r}{\Nlow_{l+1}}\prod_{j=r}^l\frac{\Nmid_j-1}{\NwedgeV_j},~~b_{l,r}=\frac{\Nlow_r\Nmid_r}{\Nlow_{l+1}\NwedgeV_r}\prod_{j=r+1}^l\frac{\Nmid_j-1}{\NwedgeV_j}.
\end{equation}
%
\end{rmk}
\begin{thm}\label{thm:5_3_5_4}
Let $P$ be a graded weighted poset of rank $d$ possessing Property UL.
Let $\alpha_j,~0\leq j \leq d-1$ be real numbers. 
Then for every $k\geq 0,~f\in C^k_0$ there exist $h_j\in C_0^j,~0\leq j\leq k$ such that
\[\nr f\nr^2=\sum_{j=0}^k\nr h_j\nr^2,\]
\[\nr U_k f\nr^2\leq \sum_{j=0}^k\left(a_{k,j}+ \sum_{i=j}^k b_{k,i} (\mu'_{i-1}-\alpha_i)_+\right)\nr h_j\nr^2,\]
where $a_{k,j}(\vec{\alpha}),b_{k,j}(\vec{\alpha})$ are as in Proposition \ref{prop:KO_5_2}.


As a consequence, for any $f\in C_0^k,$
\[\nr U_k f\nr^2\leq (\max_{j\leq k}\{a_{k,j}\}+ \sum_{i=0}^k b_{k,i} (\mu'_{i-1}-\alpha_i)_+)\nr f\nr^2,\]
and if there exist $\mu>0$ and $\alpha_j,~j\leq k$ such that for all $j,~(\mu'_{j-1}-\alpha_j)_+\leq \mu,$ then \[\nr U_k f\nr^2\leq (\max_{j\leq k}\{a_{k,j}\}+ \mu\sum_{i=0}^k b_{k,i})\nr f\nr^2.\]
If $P$ is standard regular of rank $d,$ then this statement holds with the constants $a_{k,j},b_{k,j}$ of \eqref{eq:a_b_lazy_choice}.
\end{thm}
\begin{proof}
We take $h_j,g_j~0\leq j\leq k$ as in Proposition \ref{prop:KO_5_2}. Using \eqref{eq:norm_d_k_decomp}, the proof will follow, with this choice of $(h_j)_{j}$ if we could show that
\[\sum_{i=0}^kb_{k,i}\sum_{x\in P(i-1)}\m(x)\langle (g_i)_x,(M^+_{x,0}-\alpha_i Id)(Id- M^-_{x,0})(g_i)_x\rangle_x\leq \sum_{j=0}^k\left(\sum_{i=j}^k b_{k,i} (\mu'_{i-1}-\alpha_i)_+\right)\nr h_j\nr^2.\]
From Proposition \ref{prop:bounding_correction},
\[b_{k,i}\sum_{x\in P(i-1)}\m(x)\langle (g_i)_x,(M^+_{x,0}-\alpha_i Id)(Id- M^-_{x,0})(g_i)_x\rangle_x\leq b_{k,i}(\mu'_{i-1}-\alpha_i)_+\nr g_i\nr^2.\]
Summing over $i\leq k,$ and using \eqref{eq:norm_decomp} we obtain
\begin{align*}\sum_{i=0}^kb_{k,i}&\sum_{x\in P(i-1)}\m(x)\langle (g_i)_x,(M^+_{x,0}-\alpha_i Id)(Id- M^-_{x,0})(g_i)_x\rangle_x\leq \sum_{i=0}^kb_{k,i}(\mu'_{i-1}-\alpha_i)_+\nr g_i\nr^2=\\
&=\sum_{i=0}^kb_{k,i}(\mu'_{i-1}-\alpha_i)_+(\sum_{j\leq i}\nr h_j\nr^2)=
\sum_{j=0}^{k} \left(\sum_{i=j}^k b_{k,i} (\mu'_{i-1}-\alpha_i)_+\right) \nr h_j\nr^2.\end{align*}
%
\end{proof}

\begin{ex}\label{ex:decomp_1_side}
Let $P$ be a simplicial complex, following Remark \ref{rmk:a_b_consts_nice_cases}, if we take $\alpha_j=\frac{1}{2},$ and $\mu$ as in the theorem. Note that $\mu=\frac{\lambda}{2},$ where $\lambda$ is an upper bound for the second eigenvalue of the adjacency matrices of all links.
In this case we obtain \[a_{j,i}=\frac{j-i+1}{j+2},~~ b_{j,i}=\frac{2(i+1)}{j+2}.\]
Plugging into Theorem \ref{thm:5_3_5_4} we obtain Theorem 5.2 and Corollary 5.3 of \cite{KO} (our normalization differs from theirs). In particular for any $f\in C_0^k$
\[\langle M^+_k f,f\rangle\leq\nr U_kf\nr^2 \leq (\frac{k+1}{k+2}+\frac{k+1}{2}\lambda)\nr f\nr,\]so that the largest eigenvalue of $M^+_k|_{C^k_0}$ is at most $\frac{k+1}{k+2}+\frac{k+1}{2}\lambda,$ which is \cite[Theorem 5.4]{KO}.

Let $P$ be a Grassmannian complex, following Remark \ref{rmk:a_b_consts_nice_cases} again, if we take $\alpha_j=\frac{1}{q+1},$ and $\mu$ as in the theorem. $\mu=\frac{q}{q+1}\lambda,$ where $\lambda$ is an upper bound for the second eigenvalue of all adjacency matrices of links.
In this case, $a_{j,i}=\frac{[j-i+1]_q}{[j+2]_q},$ while
$b_{j,i}=\frac{q^{j-i}(q+1)[i+1]_q}{[j+2]_q}.$
Plugging into Theorem \ref{thm:5_3_5_4} we obtain that the elements $h_j$ in the decomposition satisfy
\[\nr U_k f\nr^2\leq \sum_{j=-1}^{k-1}\left(\frac{[k-j]_q}{[k+2]_q}+ \sum_{i=j}^{k-1} \frac{q^{k-i}[i+2]_q}{[k+2]_q} \lambda\right)\nr h_j\nr^2.\]
From this we deduce that the largest eigenvalue of
$M^+_k|_{C^k_0}$ is at most $\frac{[k+1]_q}{[k+2]_q}+S(k,q)\lambda,$ where
$S(k,q)=\sum_{i=0}^k\frac{q^{k-i+1}[i+1]_q}{[k+2]_q}.$
Observe that as $q$ grows, $S(k,q)=k+1+O(\frac{1}{q})$, and $\frac{[k+1]_q}{[k+2]_q}\to \frac{1}{q}.$
%
\end{ex}
In practice, if there are bounds on $\mu'_i,$ by optimizing over the choice of the constants $\alpha_i,$ one can obtain much better bounds.

\section{The Adjacency Localization Property and its consequences}\label{sec:2sided}
In this section we describe another property, the AL property\footnote{'AL' for Adjacency Localization} for \emph{standard} weighted graded posets. This property again generalizes regularity. A poset possessing this property allows performing a localization decomposition of $\langle Af,g\rangle,$ for $f,g\in C^i.$ Moreover, some of the main results of \cite{DDFH} and \cite{AL} generalize to this setting.

\begin{definition}\label{def:for_equiv_defs}
A standard weighted graded poset $P$ of rank $d$ has property AL if for any $0\leq l\leq d-1$  
and for all $x\neq y\in C^l,$ the following equation holds
\begin{equation}\label{eq:for_equiv_1}\sum_{z\rhd x,y}\frac{\m(z)\p_{z\to x}\p_{z\to y}}{1-\p_{z\to x}}=
\sum_{z\rhd x,y}\m(z)\p_{z\to x}\p_{z\to y}\sum_{s\lhd x,y}\frac{\p_{x\to s}\p_{y\to s}}{\sum_{\substack{c\in \Ch(z\to s)\\c\neq (s,x,z)}}\p(c)}.\end{equation}
\end{definition}
\begin{rmk}\label{rmk:EL}
%
\begin{enumerate} A few comments are in place.
\item Since $P$ is standard, \eqref{eq:for_equiv_1} can also be written as
\[\sum_{z\rhd x,y}\frac{\m(z)}{(\NN(z)-1)\NN(z)}=
\sum_{z\rhd x,y}\frac{\m(z)}{\NN(z)}
\sum_{s\lhd x,y}\frac{1}{\NN(x)\NN(y)\sum_{\substack{(s,w,z)\in \Ch(z\to s)\\w\neq x}}\frac{1}{\NN(w)}}.\]
The importance of this property is that it allows a Garland-type decomposition for the adjacency operator, see Proposition \ref{prop:dec_adj} below.

\item Any regular standard poset $P$ possesses Property AL.

\item One could have tried to generalize assumption \eqref{eq:for_equiv_1} by requiring the existence of a constant $\EE_l$ such that for all $x\neq y\in C^l$
\begin{equation*}\sum_{z\rhd x,y}\frac{\m(z)\p_{z\to x}\p_{z\to y}}{1-\p_{z\to x}}=
\EE_l
\sum_{z\rhd x,y}\m(z)\p_{z\to x}\p_{z\to y}\sum_{s\lhd x,y}\frac{\p_{x\to s}\p_{y\to s}}{\sum_{\substack{c\in \Ch(z\to s)\\c\neq (s,x,z)}}\p(c)}.\end{equation*}
But by fixing $x$ and summing the LHS and RHS of this equation over all $y\neq x,$ it can be seen that $\EE_l$ must be $1.$ In the regular case this fact is equivalent to \eqref{eq:Ns_relation}.
\end{enumerate}
\end{rmk}
The name 'Adjacency Localization' is justified by the following proposition
\begin{prop}\label{prop:dec_adj}
Let $P$ be a standard graded poset having property AL. Then for any $f\in C^l,$
\begin{equation*}
\langle A_l f,f\rangle=\sum_{s\in P({l-1})}\m(s)\langle A_sf_s,f_s\rangle_s.\end{equation*}
\end{prop}
For the proof, recall that
\[\langle A_l f,f\rangle=\sum_{x\neq y\in P(l)}f(x)f(y)\sum_{z\rhd x,y}\frac{\m(z)\p_{z\to x}\p_{z\to y}}{1-\p_{z\to x}}.
\]
Now, for $s\in P(l-1),~g\in C^0(P_s)$
\begin{align*}
\langle A_{s} g,g\rangle_s&=\sum_{{x\neq y\in (P_s)(0)}}g(x)g(y)\sum_{z\rhd x,y}\frac{\m_s(z)(\p_s)_{z\to x}(\p_s)_{z\to y}}{1-(\p_s)_{z\to x}}=\\
&=\sum_{{x\neq y\in P_s(0)}}g(x)g(y)\sum_{z\rhd x,y}\frac{\m(z)}{\m(s)}\frac{\p_{z\to x}\p_{x\to s}\p_{z\to y}\p_{y\to s}}{\sum_{c\in \Ch(z\to x),~c\neq (s,x,z)}\p(c)}.
\end{align*}
Taking $g=f_s,$ summing over $s\in P({l-1})$ with weight $\m(s),$ and using \eqref{eq:for_equiv_1}, we obtain the claim.

\subsection{Sequences of spectral gaps}

In this section we will analyze the bounds on the second eigenvalue of $M^\pm_k$ that can be deduced from the knowledge, for each $l,$ of the spectral gaps of the adjacency matrices of all links in level $l.$ We follow closely the methods of \cite[Theorems 3.1, 3.2]{AL}, where again the main novelty is identifying the properties the posets need to possess in order to make the proofs generalize. We mention that these results assume lower regularity and standard weighting of the poset, while the result generalizing \cite{KO} does \emph{not} assume standard weighting.

Write $\mu_l,~k=-1,\ldots, d-2$ for an upper bound on the second eigenvalue of the adjacency matrices $A_s$ of the links of all elements $s\in P(l).$ Observe that these operators are indeed diagonalizable, since $P$ is assumed to be standard and using Observation \ref{obs:standard_under_loc}.
We will show
\begin{thm}\label{thm:generalized_Alev_Lau}
Let $P$ be a standard graded lower regular poset of rank $d$, possessing Property AL.
Then for any $0\leq l\leq d-1, -1\leq r\leq l,$ $M^+_l$ has at most $|P(r)|$ eigenvalues greater than
\[1-\prod_{j=r}^{l-1}\frac{\Nlow_{j+2}-1}{\Nlow_{j+2}}\prod_{j=r}^{l-1}(1-\mu_j).\]
In particular, \[\lambda_2(M^+_l)\leq 1-\prod_{j=-1}^{l-1}\frac{\Nlow_{j+2}-1}{\Nlow_{j+2}}\prod_{j=-1}^{l-1}(1-\mu_j).\]
\end{thm}
Since the non-zero spectrum of $M^+_k$ and $M_{k+1}^-$ is the same, including multiplicity, by Observation \ref{obs:same_spec} we obtain the same bounds for $M^-_{l+1}$ as well.
\begin{proof}
We first show that, following \cite[Lemma 3.6]{AL}, as operators
\begin{equation}\label{eq:A_L_3_6}
A_l-M^-_l\preceq\mu_{l-1}(Id|_{C^l}-M_l^-),
\end{equation}
meaning that $\mu_{l-1}(Id|_{C^l}-M_l^-)-A_l+M_l^-$ is non negative definite.
\begin{align*}
\langle (A_l-M^-_l)f,f\rangle&=\langle A_lf,f\rangle-\langle D_lf,D_lf\rangle
=\sum_{s\in P({l-1})}\m(s)(\langle A_sf_s,f_s\rangle_s-\langle f_s,\one\rangle_s^2)\\&=
\sum_{s\in P({l-1})}\m(s)(\langle A_sf_s,f_s\rangle_s-\langle \text{Pr}_{\one_s} f_s,f_s\rangle_s)=\sum_{s\in P({l-1})}\m(s)\langle (A_s-\text{Pr}_{\one_s}) f_s,f_s\rangle_s
\end{align*}
where $\text{Pr}_{\one_s}$ is the projection on the constants, and the second equality used Proposition \ref{prop:dec_adj} and
\begin{equation}\label{eq:for_3_6}D_lf(s)=D_0f_s=\langle f_s,\one\rangle_s,\end{equation} which follows from Lemma \ref{lem:D,U and constants}.
Observe that both $A_s$ and $\text{Pr}_{\one_s}$ have $\one_s$ as eigenvector for $1,$ and on its orthocomplement $\text{Pr}_{\one_s}$ vanishes. Thus,
\begin{align*}\langle (A_s-\text{Pr}_{\one_s}) f_s,f_s\rangle_s&=
\langle (A_s-\text{Pr}_{\one_s}) \text{Pr}_{\one_s}^\perp f_s,\text{Pr}_{\one_s}^\perp f_s\rangle_s=\langle A_s\text{Pr}_{\one_s}^\perp f_s,\text{Pr}_{\one_s}^\perp f_s\rangle_s\\&\quad\quad\quad\leq \lambda_2(A_s)\nr \text{Pr}_{\one_s}^\perp f_s\nr^2_s =  \lambda_2(A_s)\nr (Id-\text{Pr}_{\one_s}) f_s\nr^2_s, \end{align*}
where $\text{Pr}_{\one_s}^\perp$ is the projection to the orthocomplement of $\one_s.$
Combining with the previous equation and using the definition of $\mu_{l-1}$ we get
\begin{align*}
\langle (A_l-M^-_l)f,f\rangle&\leq
\mu_{l-1}\sum_{s\in P({l-1})}\m(s)\langle (Id-\text{Pr}_{\one_s}) f_s,(Id-\text{Pr}_{\one_s}) f_s\rangle_s\\
&=\mu_{l-1}\sum_{s\in P({l-1})}\m(s)\langle (Id-\text{Pr}_{\one_s}) f_s, f_s\rangle_s=
\mu_{l-1}\sum_{s\in P({l-1})}\m(s)(\nr f_s\nr_s^2-\langle f_s, \one\rangle_s^2)\\
&\qquad\qquad\qquad\qquad\qquad\qquad=\mu_{l-1}(\nr f\nr^2-\nr D_l f\nr^2)=\mu_{l-1}\langle (Id-M_l^-)f,f\rangle,
\end{align*}
in the first equality we used $\text{Pr}_{\one_s}$ being a projection, in the one before last we used \eqref{eq:for_3_6}, and Proposition \ref{prop:garland_posets_1}. This proves \eqref{eq:A_L_3_6}.

We prove the theorem by induction on $l.$ For $l=0,$
\[M_0^+=\frac{\Nlow_1-1}{\Nlow_1}A_0+\frac{1}{\Nlow_1}Id_{C^0}.\]
We need to show that there is at most $|P(-1)|=1$ eigenvalue greater than $\frac{\Nlow_1-1}{\Nlow_1}\mu_0+\frac{1}{\Nlow_1},$ and at most $|P(0)|$ eigenvalues greater than 
$0$ (the empty product is taken to be $1$). The latter holds since $\text{dim}(C^0)=|P(0)|,$ and the former because the eigenvalues of $M_0^+$ are of the form $\frac{\Nlow_1-1}{\Nlow_1}\gamma+\frac{1}{\Nlow_1},$ for any eigenvalue $\gamma$ of $A_0.$ Since all such $\gamma\leq 1$ and for $\gamma\leq1$
\[\gamma\leq \frac{\Nlow_1-1}{\Nlow_1}\gamma+\frac{1}{\Nlow_1},\]
only the largest eigenvalue of $M^+_0$ can be strictly greater than $\frac{\Nlow_1-1}{\Nlow_1}\mu_0+\frac{1}{\Nlow_1}$ (and that will happen precisely when $P$ is connected). The case $l=0$ follows.

Suppose we have proven the claim up to $l,$ we now prove the result for $l+1$ and any $r\leq l+1.$ The case $r=l+1$ is, as before, automatic, since $\text{dim}(C^{l+1})=|P(l+1)|.$

From \eqref{eq:A_L_3_6}
\[A_{l+1}\preceq \mu_{l}Id+(1-\mu_{l})M^-_{l+1}.\]
By Observation \ref{obs:same_spec} the non zero spectrum of $M^-_{l+1}$ is the same as the non zero spectrum of $M^+_l,$ which by induction means that, for any $r\leq l,$ $A_{l+1}$ has at most $|P(r)|$ eigenvalues greater than \[\mu_l+(1-\mu_l)\left(1-\prod_{j=r}^{l-1}\frac{\Nlow_{j+2}-1}{\Nlow_{j+2}}\prod_{j=r}^{l-1}(1-\mu_j)\right)=1-\prod_{j=r}^{l-1}\frac{\Nlow_{j+2}-1}{\Nlow_{j+2}}\prod_{j=r}^{l}(1-\mu_j).\]
For the last derivation we have used that if $A\preceq B$ then the $i$th eigenvalue of $B$ is greater or equal the $i$th eigenvalue of $A,$ see, for example, \cite{B}.
Since
\[M^+_{l+1}=\frac{\Nlow_{l+2}-1}{\Nlow_{l+2}}A_{l+1}+\frac{1}{\Nlow_{l+2}}Id,\]
we learn that for $r\leq l,$ $M^+_{l+1}$ has at most $|P(r)|$ eigenvalues greater than
\[\frac{\Nlow_{l+2}-1}{\Nlow_{l+2}}(1-\prod_{j=r}^{l-1}\frac{\Nlow_{j+2}-1}{\Nlow_{j+2}}\prod_{j=r}^{l}(1-\mu_j))+\frac{1}{\Nlow_{l+2}}=1-\prod_{j=r}^{l}\frac{\Nlow_{j+2}-1}{\Nlow_{j+2}}\prod_{j=r}^{l}(1-\mu_j).\]
\end{proof}

\begin{ex}\label{ex:AL_Grass}
When $P$ is a simplicial complex, we obtain Theorems 3.1, 3.2 of \cite{AL}.

In the case of the Grassmannian poset we obtain new results. For example, take $r=-1$
then
\[\nr U_k f\nr^2\leq
\sum_{j=-1}^{k-1}(1-q^{k-j}\frac{[j+2]_q}{[k+2]_q}\prod_{i=j}^{k-1}(1-\mu_i))\nr h_j\nr^2\leq (1-\frac{q^{k+1}}{[k+2]_q}\prod_{i=-1}^{k-1}(1-\mu_i))\nr f\nr^2.
\]
For $q>>1,$ $q^{k+1}/[k+2]_q\to 1-O(\frac{1}{q}),$ giving the bound $1-(1-O(q^{-1}))\prod_{j=-1}^{k-1}(1-\mu_j).$
\end{ex}

\subsection{Expanding two-sided local spectral posets and a decomposition theorem}


In \cite[Section 8]{DDFH} an alternative definition of expanding weighted graded posets was given using a global characterization based on comparing the upper and lower random walks. This global expansion definition is called $\lambda$-eposet (See Definition \ref {def:eposet}). It was moreover shown in \cite{DDFH} that for standard weighted simplicial complexes global expansion (i.e., being $\lambda$-eposets) coincides with the $O(\lambda)$ two-sided local spectral expansion.


We now provide a more general criterion on posets under which the two definitions agree. We should stress that this criterion may not be the most general one for obtaining this equivalence.

We recall the notation $\lambda(M)$ for the maximum between the second largest eigenvalue of the matrix $M$, and the absolute value of the smallest eigenvalue of $M.$

\begin{prop}\label{prop:equiv_for_2_sided}
\begin{enumerate}
Let $P$ be a standard graded poset of rank $d$ which possesses Property AL.
\item Suppose that there exists $\lambda$ such that for any $x\in P(\leq d-2),$ \[\lambda(A_{x})\leq \lambda,\]where $A_x$ is the $0-$th adjacency matrix of $P_x.$ Then in operator norm, for any $0\leq l\leq d-1,$
\[\nr A_l-
U_{l+1}D_l \nr\leq 
\lambda.\]
\item
Suppose that $P$ is lower regular with constants $(\Nlow_l)_l,$ and that for any $x,y$ there is at most one $z$ which covers both. Then if
$P$ satisfies \[\nr A_l-
U_{l+1}D_l \nr\leq \lambda,\]
then for any $x\in P$ it holds that \[\lambda(A_x)\leq
\max_{l\leq d-1}\{2\Nlow_l\}\lambda
\]
\end{enumerate}
\end{prop}
We note again that since we assume that $P$ is standard, and by applying Observation \ref{obs:standard_under_loc}, it makes sense to consider the eigenvalues of the link adjacency matrices.
\begin{proof}
Our proof modifies the argument from \cite[Theorem 5.5]{DDFH}.
We start with the first item. It is enough to show that for any $f\in C^l$
\[|\langle A_l f,f\rangle-\langle U_{l+1}D_lf,f\rangle|=|\langle A_l f,f\rangle-\nr D_lf\nr^2|\leq \lambda\nr f\nr^2.\]
Using Proposition \ref{prop:dec_adj} and \eqref{eq:for_3_6} we obtain
\begin{align*}
|\langle A_l f,f\rangle-\nr D_lf\nr^2|&=|\sum_{s\in P({l-1})}\m(s)(\langle A_sf_s,f_s\rangle_s-(D_lf(s))^2)|=\\
&=|\sum_{s\in P({l-1})}\m(s)(\langle A_sf_s,f_s\rangle_s-\langle f_s,\one\rangle_s^2)|\leq\sum_{s\in P({l-1})}\m(s)|\langle A_sf_s,f_s\rangle_s-\langle f_s,\one\rangle_s^2|\leq\\
&\leq \lambda \sum_{s\in P({l-1})}\m(s)\nr f_s\nr^2=\lambda\nr f\nr^2,
\end{align*}where the last equality used Proposition \ref{prop:garland_posets_1}.

For the second item, we want to show that under the assumptions, for any $s\in P(\leq d-2),$
\[\lambda(A_s)\leq \max_l\{{2\Nlow_l}\}\lambda.\]
We will use induction on $l.$
When $l=-1,$ since $U_{-1}D_0$ is just the projection on constant functions,
\[\nr A_\smallest-UD\nr\leq \lambda.\]
Assume the claim for all $l'<l-1,$ take $s\in P({l-1}),$ and $f\in C^0(P_s),~f\perp \one.$
We wish to bound
\[\left|\frac{\langle A_sf,f\rangle}{\langle f,f\rangle}\right|\leq2\Nlow_l\lambda\]
To this end define the lifting $f^s\in C^l$ by
$f^s(x)=f(x),~x\in P_{s}(0),~\text{and otherwise }f^s(x)=0,$ for $x\in P_{s}(0),$ where we identify $P_{s}(0)$ as a subset of $P(l).$ We scale $f,f^s$ so that $\nr f^s\nr=1.$
This means that
\[\nr f\nr^2_s=\sum_{x\in P_s(0)}\m_s(x)f(x)^2=\sum_{x\in P(l)}\frac{\m(x)\p_{x\to s}}{\m(s)}f(x)^2.\]
Since $P$ is lower regular, we obtain
\[\nr f\nr^2_s=\frac{1}{\Nlow_l\m(s)}\nr f^s\nr^2=\frac{1}{\Nlow_l\m(s)}.\]
From Proposition \ref{prop:dec_adj}
\[\langle A f^s , f^s\rangle = \sum_{s'\in P({l-1})}\m(s')\langle A_{s'} f^s_{s'},f^{s}_{s'}\rangle_{s'}.\]
By our assumption that for any $s\neq s'$ there exists at most one common element which covers them, $f^s_{s'},$ for $s\neq s'$ is supported on at most one element, hence for $s'\neq s$
\[\langle A_{s'}f^s_{s'},f^{s}_{s'}\rangle_{s'}=0\]
Thus, we have $\langle A f^s,f^s\rangle = \m(s)\langle A_sf,f\rangle_s.$
We will show
\begin{equation}\label{eq:Their_5_6}
\langle UD f^s,f^s\rangle\leq\lambda\langle f^s,f^s\rangle.
\end{equation}

Using it,
\begin{align*}
\left|\frac{\langle A_s f,f \rangle_s}{\nr f\nr^2_s}\right|=\Nlow_l\m(s)|\langle A_s f,f \rangle_s|&={\Nlow_l}|\langle A f^s,f^s\rangle|\\&\leq
{\Nlow_l}(|\langle (A -UD) f^s,f^s\rangle|+|\langle UD f^s,f^s\rangle|)\leq {2\Nlow_l}\lambda\nr f^s\nr^2={2\Nlow_l}\lambda.
\end{align*}

It is left to prove Equation \eqref{eq:Their_5_6}.
\begin{align*}&\langle UDf^s,f^s\rangle=\sum_{x,y\in P(l)}\m(x)f^s(x)\m(y)f^s(y)\sum_{s'\lhd x,y}\frac{\p_{x\to s'}\p_{y\to s'}}{\m(s')}=
\\&=\frac{1}{\m(s)}\sum_{x,y\in P_s(0)}\m(x)f(x)\p_{x\to s}\m(y)f(y)\p_{y\to s}+
\sum_{x,y\in P(l)}\m(x)f^s(x)\m(y)f^s(y)\sum_{\substack{s'\neq s\\s'\lhd x,y}}\frac{\p_{x\to s'}\p_{y\to s'}}{\m(s')}.
\end{align*}
Since, $f\perp \one_s$ the first summand in the right hand side is
\[\m(s)\left(\sum_{x\in P_s(0)}\m_s(x)f(x)\right)^2=0.\]

For the second summand, any $s'\neq s$ has at most one $x$ which covers both, and $x$ is the only element which cover $s'$ for which $f^s$ does not vanish.
Thus, the second sum is
\[\sum_{x,y\in P(l)}\m(x)f^s(x)\m(y)f^s(y)\sum_{s'\neq s\lhd x,y}\frac{\p_{x\to s'}\p_{x\to s'}}{\m(s')}=
\sum_{x\rhd s}\m^2(x)f(x)^2\sum_{\substack{s'\neq s\\
s'\lhd x}}\frac{\p_{x\to s'}^2}{\m(s')}\]

Equation \eqref{eq:Their_5_6} will follow if we can show\begin{equation}\label{eq:final_for_2_sied}\m^2(x)\sum_{\substack{s'\neq s\\
s'\lhd x}}\frac{\p_{x\to s'}^2}{\m(s')}\leq \lambda\m(x).\end{equation}
Indeed,
\[\sum_{x\rhd s}\m^2(x)f(x)^2\sum_{\substack{s'\neq s\\
s'\lhd x}}\frac{\p_{x\to s'}^2}{\m(s')}\leq
\lambda\sum_{x\rhd s}\m(x)f(x)^2=\lambda\nr f^s\nr^2=\lambda,\]
proving the \eqref{eq:Their_5_6}.

For \eqref{eq:final_for_2_sied}, note that\[\m^2(x)\sum_{\substack{s'\neq s\\
s'\lhd x}}\frac{\p_{x\to s'}^2}{\m(s')}\leq\m^2(x)\sum_{\substack{s'\lhd x}}\frac{\p_{x\to s'}^2}{\m(s')}=\langle UD\delta_x,\delta_x\rangle,\]
where $\delta_x\in C^l$ is the Kronecker delta function at $x.$ Observe that $\nr\delta_x\nr^2=\m(x).$ Finally,
\[\langle UD\delta_x,\delta_x\rangle=\langle (UD-A_l)\delta_x,\delta_x\rangle+\langle A_l\delta_x,\delta_x\rangle=\langle (UD-A_l)\delta_x,\delta_x\rangle\leq\lambda\nr\delta_x\nr^2=\lambda\m(x),\]
where the second equality uses that $A_l$ has $0$ diagonal, and the inequality is by the assumptions.
\end{proof}

A slightly generalized form of Lemma 4.5 of \cite{DDFH}, with essentially the same proof, is
\begin{lemma}\label{lem:properness}
If $P$ is a $\mu$-global-eposet with
$\mu<r_j,$ then $U_j$ is injective.
\end{lemma}
\begin{proof}
It is enough to show that for any non zero $f\in C^l,$
\[\langle Uf,Uf\rangle=\langle f,DU f\rangle >0,\]
(we omit the subscripts for convenience).
Now,
\[\langle f,DU f\rangle=\langle f,(DU-r_lId-\delta_lUD) f\rangle+r\nr f\nr^2+\delta\langle f,UD f\rangle\geq \langle f,(DU-r_lId-\delta_lUD) f\rangle+r\nr f\nr^2,\]
where the inequality is since $\langle f,UD f\rangle=\langle Df,D f\rangle\geq 0.$
By assumption, $\langle f,(DU-r_lId-\delta_lUD) f\rangle\geq- \mu\nr f\nr^2.$ Putting together we obtain
\[\langle Uf,Uf\rangle\geq (r-\mu)\nr f\nr^2>0.\]

\end{proof}

\cite{DDFH} also proved
\begin{thm}[Theorem 5.5 of \cite{DDFH}]
If a simplicial complex is a 2-sided $\lambda$ high dimensional expander, then it is also a $\lambda-$global high dimensional expander, with constants $\delta_j=1-\frac{1}{j+2}, r_j=\frac{1}{j+2}.$ If a $d-$dimensional simplicial complex is a $\mu-$global high dimensional expander (with the previous constants) then it is also a two-sided $3d\mu$-local spectral expander.
\end{thm}
Based on the previous proposition, the relation between the upper walk and the adjacency operator in the lower regular case, and Remark \ref{rmk:EL}, we deduce the following generalization, which covers, in particular the simplicial complex case, the Grassmannian poset case, and many more.
\begin{thm}\label{thm:equiv_for_2_sided}
Suppose $P$ is a standard lower regular poset of rank $d$ possessing Property AL.
\begin{itemize}
    \item Then if $P$ is a two-sided $\lambda-$local spectral expanding poset, then it is a $(\max_{l\leq d-1}\{1-\frac{1}{\Nlow_{l+1}}\}\lambda)$-global eposet, with constants $r_l=\frac{1}{\Nlow_{l+1}}, \delta_l=1-\frac{1}{\Nlow_{l+1}}.$ 
\item If $P$ is a $\mu$-global eposet with constants $r_l=\frac{1}{\Nlow_{l+1}}, \delta_l=1-\frac{1}{\Nlow_{l+1}},~l\leq d-1,$ with the additional property that for any two $x\neq y\in P$ there is at most one element covering both, then it is a two-sided $(2\max_{l\leq d-1}\{(1+\frac{1}{\Nlow_{l+1}-1}){\Nlow_{l}}\}\mu)-$local spectral expanding poset.
\end{itemize}
\end{thm}\label{thm:DD_decomp}

$\mu-$global eposets have a decomposition theorem which refines that of Theorem \ref{thm:5_3_5_4}:
\begin{thm}[Theorem 8.6 of \cite{DDFH}]
Let $P$ be a rank $d$ $\mu-$global eposet with parameters $(r_i)_i,(\delta_i)_i.$ Suppose $\mu$ is small enough (as a function of $d$ and the parameters). Then 
for any $f\in C^l$ there is a unique decomposition
\[f= f_{-1}+f_0+\ldots+f_l,~f_{-1}\in U^{l+1}(C^{-1}),~f_i\in U^{l-i}(\text{ker}(D_i)),~i\geq 0,\]
with the following properties (all constants in the $O$ notation will depend on $d$ and the parameters):
\begin{enumerate}
    \item (\emph{Near orthogonality}) For any $i\neq j,$
    \[\langle f_i,f_j\rangle=O(\mu),~\nr f\nr^2=(1+O(\mu))(\nr f_{-1}\nr^2+\ldots+\nr f_l\nr^2).\]
    \item (\emph{Near Eigenvalues}) If $l<d,$ it holds that
    \[\nr DUf_i-r^l_{l-i+1}f_i\nr\leq O(\mu)\nr f_i\nr,~~\langle DUf,f\rangle=\sum r^l_{l-i+1}\nr f_i\nr^2+O(\mu)\nr f\nr^2,\]
    where $r^l_{l+2}=1$ and for $i\geq 0$
    $r^l_i=r_l+\sum_{j=l-i}^{l-1}(\prod_{h=j+1}^l\delta_h)r_j.$
\end{enumerate}
\end{thm}
Combining Theorem \ref{thm:DD_decomp}, Lemma \ref{lem:properness} and Theorem \ref{thm:equiv_for_2_sided} immediately yields
\begin{thm}\label{thm:2_sided_decomp_DD_style}
Suppose $P$ is a standard, lower regular poset, which possesses Property AL and is $2-$sided $\lambda-$link expanding poset, for small enough $\lambda,$ then each $U^l$ is injective, and any $f\in C^l$ has a decomposition as in Theorem \ref{thm:DD_decomp} where the constants are $r_l=\frac{1}{\Nlow_{l+1}}, \delta_l=1-\frac{1}{\Nlow_{l+1}}.$ 
\end{thm}
\begin{ex}%
As above, when $P$ is regular Property AL holds automatically. Thus, we can examine this theorem for simplicial complexes and Grassmannian posets. In these two cases we have (\cite[Section 8]{DDFH}):
\\In the simplicial complex case, for the complete complex on $n$ vertices, \[r_i=\frac{1}{i+2},~\delta_i=1-\frac{1}{i+2},~r^l_i=\frac{i}{l+2},~\mu=O(\frac{1}{n-i}).\]
In the Grassmannian poset case, for the complete Grassmannian of $V=\F_q^n,$
\[r_i=\frac{q-1}{q^{i+2}-1},\delta_i=1-\frac{q-1}{q^{i+2}-1},~r^l_i=1-\prod_{j=l-i+1}^l(1-\frac{q-1}{q^{j+2}-1}),~\mu=O(\frac{1}{q^{n-i}}).\]
\end{ex}

\section{The Trickling Down Localization Property and its applications}\label{sec:trickl}
In \cite{O} a beautiful and useful criterion for bounding the spectral gap of the adjacency matrix of a complex using the corresponding gaps in the links ('trickling down') was given. In this section we generalize it to more general posets. We will restrict our attention to standard graded weighted posets. In this case adjacency matrices will be self adjoint.

Property TL\footnote{'TL' for Trickling Localization} of weighted graded posets, to be defined next, generalizes $\wye$ regularity, and will imply a general Trickling down phenomenon.
\begin{definition}\label{def:TL}
Let $P$ be a graded weighted poset.
\begin{enumerate}
\item\textbf{TL.1 }$P$ possesses Property TL.1 if $\exists \csame>0 ~\text{such that }\forall y\in P(0),$
\begin{align}\label{eq:prop_1_for_norm_sqr_hat_f}
&\csame \m(y)=\sum_{z\rhd y}\m(z)\p_{z\to y}^2\sum_{x\lhd z,~x\neq y}\frac{\p_{z\to x}}{(1-\p_{z\to x})^2}.
\end{align}
\item\textbf{TL.2 }$P$ possesses Property TL.2 if $\exists \cdiff>0~\text{such that }\forall y_1\neq y_2\in P(0),$
\begin{align}\label{eq:prop_2_for_norm_sqr_hat_f}
\cdiff\sum_{z\rhd y_1,y_2}\frac{\m(z)\p_{z\to y_1}\p_{z\to y_2}}{1-\p_{z\to y_1}}=\sum_{z\rhd y_1,y_2}\m(z)\p_{z\to y_1}\p_{z\to y_2}\sum_{x\lhd z,~x\neq y_1,y_2}\frac{\p_{z\to x}}{(1-\p_{z\to x})^2}.
\end{align}
\item\textbf{TL.3 }$P$ possesses Property TL.3 if $\exists \csametwo>0 ~\text{such that }\forall y\in P(0),$
\begin{equation}\label{eq:prop_1_for_loc_of_adj}
\csametwo \m(y)=
\sum_{\substack{P(0)\ni x\neq y,\\
z\neq w\in P(1),\\z,~w\rhd x,y,\\
u\in P(2),~u\rhd z,w}}\frac{\m(u)\p_{u\to z}\p_{u\to w}\p_{w\to y}\p_{z\to y}\p_{w\to x}\p_{z\to x}}
{(\sum_{c\in \Ch(u\to x)}\p(c))(1-\frac{\p_{u\to z}\p_{z\to x}}{\sum_{c\in \Ch(u\to x)}\p(c)})(1-\p_{w\to x})(1-\p_{z\to x})}.\\
\end{equation}
\item\textbf{TL.4 }$P$ possesses Property TL.4 if $\notag\exists \cdifftwo>0 ~\text{such that }\forall y_1\neq y_2\in P(0),$\begin{align}\label{eq:prop_2_for_loc_of_adj}
\notag&\cdifftwo\sum_{z\rhd y_1,y_2}\frac{\m(z)\p_{z\to y_1}\p_{z\to y_2}}{1-\p_{z\to y_1}}=\\
&=\sum_{\substack{P(0)\ni x\neq y_1,y_2,\\
z\neq w\in P(1),\\z\rhd x,y_2,w\rhd x,y_1,\\
u\in P(2),~u\rhd z,w}}\frac{\m(u)\p_{u\to z}\p_{u\to w}\p_{w\to y_1}\p_{z\to y_2}\p_{w\to x}\p_{z\to x}}
{(\sum_{c\in \Ch(u\to x)}\p(c))(1-\frac{\p_{u\to z}\p_{z\to x}}{\sum_{c\in \Ch(u\to x)}\p(c)})(1-\p_{w\to x})(1-\p_{z\to x})}.
\end{align}
\end{enumerate}
$P$ possesses Property TL, if it satisfies TL.1-TL.4.
\end{definition}
\begin{rmk}\label{rmk:TL_alternative}
Since $P$ is standard, Properties TL.1-4 can also be written in equivalent forms
\begin{itemize}
    \item \textbf{TL.1 }$\exists \csame>0 ~\text{such that }\forall y\in P(0),~\csame \m(y)
=\sum_{z\rhd y}\frac{\m(z)}{\NN(z)(\NN(z)-1)}.
$
\item \textbf{TL.2 }
$\exists \cdiff>0 ~\text{such that }\forall y_1\neq y_2\in P(0),$
\begin{align*}
\cdiff\sum_{z\rhd y_1,y_2}\frac{\m(z)}{\NN(z)(\NN(z)-1)}=\sum_{z\rhd y_1,y_2}\frac{\m(z)(\NN(z)-2)}{\NN(z)(\NN(z)-1)^2}.
\end{align*}
\item\textbf{TL.3 }$\exists \csametwo>0 ~\text{such that }\forall y\in P(0),$
\begin{align*}
\csametwo \m(y)=
\sum_{\substack{P(0)\ni x\neq y,\\
z\neq w\in P(1),\\z,~w\rhd x,y,\\
u\in P(2),~u\rhd z,w}}
\frac{\m(u)}{\NN(u)\NN(w)(\NN(w)-1)\NN(z)(\NN(z)-1)(\sum_{v\neq z, x\lhd v\lhd u}\frac{1}{\NN(v)})}.\end{align*}
\item\textbf{TL.4 }$\exists \cdifftwo>0 ~\text{such that }\forall y_1\neq y_2\in P(0),$
\begin{align*}\label{eq:prop_2_for_loc_of_adj}
\notag&\cdifftwo\sum_{z\rhd y_1,y_2}\frac{\m(z)}{\NN(z)(\NN(z)-1)}=\\
&=
\sum_{\substack{P(0)\ni x\neq y_1,y_2,\\
z\neq w\in P(1),\\z\rhd x,y_2,w\rhd x,y_1,\\
u\in P(2),~u\rhd z,w}}\frac{\m(u)}{\NN(u)\NN(w)(\NN(w)-1)\NN(z)(\NN(z)-1)(\sum_{v\neq z, x\lhd v\lhd u}\frac{1}{\NN(v)})}.
\end{align*}

\end{itemize}
\end{rmk}
The constants of Property TL are not independent, as the following lemma shows.
\begin{lemma}\label{lem:cdiff_plus_csame}
If $P$ possesses Properties TL.1, TL.2 then $\csame+\cdiff=1.$
If $P$ possesses Properties TL.3, TL.4 then
$\csametwo+\cdifftwo=1.$
\end{lemma}
\begin{proof}
We begin with the first equation. Fix $y.$ We will show that summing the RHS of TL.2, with $y=y_1$ over all $y_2\neq y,$ plus the RHS of TL.1, gives $\m(y),$ while performing the same sum over the LHSs gives $(\csame+\cdiff)\m(y).$
For the RHS, since each $z$ which covers $y$ covers $\NN(z)-1$ other elements,
\[\sum_{y_2\neq y}\sum_{z\rhd y,y_2}\frac{\m(z)(\NN(z)-2)}{\NN(z)(\NN(z)-1)^2}=\sum_{z\rhd y}\frac{\m(z)(\NN(z)-2)}{\NN(z)(\NN(z)-1)}.\]
Adding $\sum_{z\rhd y}\frac{\m(z)}{\NN(z)(\NN(z)-1)}$ we get $\sum_{z\rhd y}\frac{\m(z)}{\NN(z)}=\m(y),$ by \eqref{eq:weight}.
For the LHS,
\[\cdiff\sum_{y_2\neq y}\sum_{z\rhd y,y_2}\frac{\m(z)}{\NN(z)(\NN(z)-1)}=\cdiff\sum_{z\rhd y}\frac{\m(z)}{\NN(z)}=\m(y).\]

For the second equation, as we saw in the first equation, summing the LHS of TL.3 over all $y,$ and the LHS of TL.4 over all $y_1\neq y_2$ gives
\[\sum_{y\in P(0)}(\csametwo+\cdifftwo)\m(y)=\csametwo+\cdifftwo.\]
Summing the RHS of TL.3 over all $y,$ and the RHS of TL.4 over all $y_1,y_2$ gives
\begin{align*}
\sum_{y_1,y_2\in P(0)}&\sum_{\substack{P(0)\ni x\neq y_1,y_2,\\
z\neq w\in P(1),\\z\rhd x,y_2,w\rhd x,y_1,\\
u\in P(2),~u\rhd z,w}}\frac{\m(u)}{\NN(u)\NN(w)(\NN(w)-1)\NN(z)(\NN(z)-1)(\sum_{v\neq z, x\lhd v\lhd u}\frac{1}{\NN(v)})}\\
&=\sum_{\substack{x\in P(0)\\
z\neq w\in P(1),\\z,w\rhd x\\
u\in P(2),~u\rhd z,w,\\
y_1,y_2\neq x,w\rhd y_1,z\rhd y_2}}\frac{\m(u)}{\NN(u)\NN(w)(\NN(w)-1)\NN(z)(\NN(z)-1)(\sum_{v\neq z, x\lhd v\lhd u}\frac{1}{\NN(v)})}\\
&=\sum_{\substack{x\in P(0)\\
z\neq w\in P(1),\\z,w\rhd x\\
u\in P(2),~u\rhd z,w}}\frac{\m(u)}{\NN(u)\NN(w)\NN(z)(\sum_{v\neq z, x\lhd v\lhd u}\frac{1}{\NN(v)})}\\
&=\sum_{\substack{x\in P(0)\\
z\in P(1),~z\rhd x\\
u\in P(2),~u\rhd z,\\
w\neq z, x\lhd w\lhd u}}\frac{\m(u)}{\NN(u)\NN(z)(\sum_{v\neq z, x\lhd v\lhd u}\frac{1}{\NN(v)})}\frac{1}{\NN(w)}\\
&=\sum_{\substack{x\in P(0)\\
z\in P(1),~z\rhd x\\
u\in P(2),~u\rhd z}}\frac{\m(u)}{\NN(u)\NN(z)(\sum_{v\neq z, x\lhd v\lhd u}\frac{1}{\NN(v)})}\sum_{w\neq z, x\lhd w\lhd u}\frac{1}{\NN(w)}\\
&=\sum_{\substack{x\in P(0)\\
z\in P(1),~z\rhd x\\
u\in P(2),~u\rhd z}}\frac{\m(u)}{\NN(u)\NN(z)}=\sum_{x\in P(0)}\m(x)=1,
\end{align*}
the one before last equality is Observation \ref{obs:properties_of_weights}.
\end{proof}
The TL property generalizes structural properties, as the following lemma shows:
\begin{lemma}\label{lem:xzw_implies_decomp_Ahf}
Let $P$ be a standard graded weighted poset, lower regular at level $1$ with constant $\Nlow_1.$
Then it satisfies properties \eqref{eq:prop_1_for_norm_sqr_hat_f},~\eqref{eq:prop_2_for_norm_sqr_hat_f}
with constants
\[\csame = \frac{1}{\Nlow_1-1},~~\cdiff=\frac{\Nlow_1-2}{\Nlow_1-1}.
\]

If, $P$ is also lower regular at level $2,$ with constant $\Nlow_2,$ middle regular at level $1$ with constant $\Nmid_1,$ and $\wye$ regular with constant $\Rwye,$ then it possesses Property TL
with additional constants
\[\csametwo=\frac{(\frac{\Nlow_2\Nlow_1}{\Nmid_1}-1)\Rwye(\Rwye-1)}
{(\Nmid_1-1)\Nmid_1({\Nlow_1}-1)^2},~\cdifftwo=\frac{(\frac{\Nlow_2\Nlow_1}{\Nmid_1}-2)\Rwye-(\Nlow_1-2)}{(\Nmid_1-1)({\Nlow_1}-1)}.\]
\end{lemma}
\begin{proof}
For TL.1, summing the LHS of TL.1 over all $y$ gives $\csame.$ Summing the RHS yields
\[\sum_{y\in P(0)}\sum_{z\rhd y}\frac{\m(z)}{\Nlow_1(\Nlow_1-1)}=\sum_{z\in P(1)}\sum_{y\lhd z}\frac{\m(z)}{\Nlow_1(\Nlow_1-1)}=\sum_{z\in P(1)}\frac{\m(z)}{\Nlow-1}=\frac{1}{\Nlow-1}.\]
TL.2 now follows from the result for TL.1 together with Lemma \ref{lem:cdiff_plus_csame}.

For TL.3, using standarndess, lower regularity at level $1,2$ and middle regularity at level $1,$ the RHS is
$\sum_{y<u\in P(2)}\sum_{\substack{P(0)\ni x\neq y,\\
z\neq w\in P(1),\\u\rhd z,w\rhd x,y}}\frac{\m(u)}
{(\Nmid_1-1)\Nlow_2\Nlow_1({\Nlow_1}-1)^2}.$ Note that every $P(0)\ni y<u\in P(2)$, \[|\{(x,z,w)\in P(0)\times P(1)\times P(1)|x\neq y,
z\neq w,u\rhd z,w\rhd x,y\}|=(\frac{\Nlow_2\Nlow_1}{\Nmid_1}-1)\Rwye(\Rwye-1).\]Indeed, using \eqref{eq:2_desc}, whose assumptions are met, the number of different $x\in P(0),$ different from $y,$ which descend from $u$ is $\frac{\Nlow_2\Nlow_1}{\Nmid_1}-1.$ For each such $x$ there are $\Rwye$ ways to choose $z,$ and then $\Rwye-1$ ways to pick $w,$
from the definition of $\Rwye.$
Thus, \[\sum_{y<u\in P(2)}\sum_{\substack{P(0)\ni x\neq y,\\
z\neq w\in P(1),\\u\rhd z,w\rhd x,y}}\frac{\m(u)}
{(\Nmid_1-1)\Nlow_2\Nlow_1({\Nlow_1}-1)^2}=\sum_{y<u\in P(2)}\frac{(\frac{\Nlow_2\Nlow_1}{\Nmid_1}-1)\Rwye(\Rwye-1)\m(u)}
{(\Nmid_1-1)\Nlow_2\Nlow_1({\Nlow_1}-1)^2}.\]
Applying Observation \ref{obs:properties_of_weights}, in the case of lower regularities at levels $1,2$ and middle regularity at level $1,$ \eqref{eq:2_desc} and the definition of $\csametwo$ we get:
\begin{align*}
\sum_{y<u\in P(2)}\frac{(\frac{\Nlow_2\Nlow_1}{\Nmid_1}-1)\Rwye(\Rwye-1)\m(u)}
{(\Nmid_1-1)\Nlow_2\Nlow_1({\Nlow_1}-1)^2}=
\csametwo\sum_{y<u\in P(2)}\frac{\Nmid_1\m(u)}
{\Nlow_2\Nlow_1}=\csametwo\m(y),\end{align*}

Turning to TL.4, first note that for every $P(0)\ni y_1\neq y_2<u\in P(2)$ we have \begin{align*}\{(x,z,w)\in P(0)\times P(1)\times P(1)|x\neq y_1,y_2,
z\neq w,z\rhd x,y_2,&~w\rhd x,y_1,~u\rhd z,w\}=\\&=(\frac{\Nlow_2\Nlow_1}{\Nmid_1}-2)(\Rwye)^2-(\Nlow_1-2)\Rwye.\end{align*}
Using \eqref{eq:2_desc}, there are $\frac{\Nlow_2\Nlow_1}{\Nmid_1}-2$ to choose $x\neq y_1,y_2$ which satisfies $x<u.$ For each such $x,$ if we ignore the requirement $z\neq w,$ there are $(\Rwye)^2$ ways to choose $z,w.$ We need to subtract the cases where $z=w.$ The number of these cases is precisely $\Rwye(\Nlow_1-2),$ since we can first choose $z$ which covers both $y_1,y_2$ and then choose $x$ from the remaining $\Nlow_1-2$ elements covered by $z.$
Second, from Equation \ref{eq:weight} it follows that
\[\sum_{z>y_1,y_2}\m(z)=\sum_{z\rhd y_1,y_2}\sum_{u\rhd z}\frac{\m(u)}{\Nlow_2}=\frac{1}{\Nlow_2}\sum_{P(2)\ni u>y_1,y_2}\m(u)|\{z|~z\lhd u,y_1,y_2\lhd z\}|=\frac{\Rwye}{\Nlow_2}\sum_{P(2)\ni u>y_1,y_2}\m(u).\]
Combining these two observations, we see that the coefficient of $f(y_1)g(y_2)$ for $y_1\neq y_2$ is  \begin{align*}&\frac{(\frac{\Nlow_2\Nlow_1}{\Nmid_1}-2)(\Rwye)^2-(\Nlow_1-2)\Rwye}{(\Nmid_1-1)\Nlow_2\Nlow_1({\Nlow_1}-1)^2}\sum_{y_1,y_2<u\in P(2)}\m(u)=\\\qquad\qquad\qquad\qquad
&=\frac{\Nlow_2\left((\frac{\Nlow_2\Nlow_1}{\Nmid_1}-2)(\Rwye)^2-(\Nlow_1-2)\Rwye\right)}{\Rwye(\Nmid_1-1)\Nlow_2\Nlow_1({\Nlow_1}-1)^2}\sum_{z\lhd y_1,y_2}\m(z)=
\cdifftwo\sum_{z\lhd y_1,y_2}\frac{\m(z)}{\Nlow_1(\Nlow_1-1)}.\end{align*}
\end{proof}
Recall that the adjacency matrix $A_l:C^l\to C^l$  \[
A_lf(x)=\frac{1}{\m(x)}\sum_{\substack{y\neq x,\\z\rhd x,y}}\frac{\m(z)\p_{z\to y}\p_{z\to x}}{1-\p_{z\to x}}f(y).
\]is self adjoint when $P$ is standard, and is thus diagonalizable with an orthonormal basis of eigenvectors. When $l=0$ we shall omit $l$ from notations and denote the adjacency matrix by $A,$ or by $A_s$ if we consider the $0$-th adjacency matrix of the link of $s.$ We also recall that $A$ always has the trivial eigenvalue $1$ with the eigenvector $\one.$ Eigenvalues for eigenvectors of $A$ which are orthogonal to $\one$ are called non trivial.

Before we turn to the Trickling Down theorem, let us examine what kind of localization can be deduced from Property TL. To this end, for a given standard, weighted, graded poset, define, for $f\in C^0,$ and $x\in P(0)$ the induced function $\hf_x\in C^0(P_x)$ on the link $P_x$ by
\begin{equation}\label{eq:hat_f}\hf_x(z)=\frac{1}{1-\p_{z\to x}}\sum_{y\neq x,~y\lhd z}\p_{z\to y}f(y).
\end{equation}
\begin{rmk}
The two induced functions $\hf_x,f_x$ are very different in nature. In particular, for $x\in P(0),$ $f\in C^1$ induces $f_x\in C^0_x.$ On the other hand, $f\in C^0$ induces $\hf_x$ in $C^0_x.$ The trickling down phenomenon builds on the properties of $\hf_x,$ for $P$ possessing Property TL.
\end{rmk}
\begin{prop}\label{prop:trickling_localization1}
Let $P$ be a standard weighted graded poset, and $\hf_x$ as above.
The function $\hf$ has the following properties.
\begin{enumerate}
\item\label{eq:f_x_1_x}
$\langle \hf_x,\one_x\rangle_x=Af(x).
$
\item Suppose $P$ possesses Properties TL.1, TL.2, then we have
\begin{equation}\label{eq:norm_sqr_hat_f}
\sum_{x\in P(0)}\m(x)\langle \hf_x,\hg_x\rangle=
\csame \langle f,g\rangle + {\cdiff}\langle Af,g\rangle.
\end{equation}
\item If $P$ possesses Properties TL.3, TL.4 then
\begin{equation}\label{eq:loc_dec_of_adj}
\langle Af,g\rangle = \frac{1}{\cdifftwo}
\left(\sum_{x\in P(0)}\m(x)\langle A_x\hf_x,\hg_x\rangle_x-
\csametwo \langle f,g\rangle\right).
\end{equation}

\end{enumerate}
\end{prop}
\begin{proof}
For the first item,
\begin{align*}
\langle \hf_x,\one_x\rangle_x&=\sum_{z\in P_x(0)}\m_x(z)\hf_x(z)=\sum_{z\in P_x(0)}\frac{\m(z)\p_{z\to x}}{\m(x)(1-\p_{z\to x})}\sum_{\substack{y\neq x\\y\lhd z}}\p_{z\to y}f(y)=Af(x).
\end{align*}
For the second item, first note,
\begin{align*}
\sum_{x\in P(0)}\m(x)\langle\hf_x,\hg_x\rangle&=
\sum_{x\in P(0)}\m(x)\sum_{x\lhd z}\frac{\m(z)\p_{z\to x}}{\m(x)(1-\p_{z\to x})^2}\left(\sum_{y\neq x,~y\lhd z}\p_{z\to y}f(y)\right)\left(\sum_{y\neq x,~y\lhd z}\p_{z\to y}g(y)\right)=\\
&=\sum_{y_1,y_2\in P(0)}f(y_1)g(y_2)\sum_{z\rhd y_1,y_2}\m(z)\p_{z\to y_1}\p_{z\to y_2}\sum_{x\lhd z,~x\neq y_1,y_2}\frac{\p_{z\to x}}{(1-\p_{z\to x})^2}=\\
&= \sum_{y\in P(0)}f(y)g(y)\sum_{z\rhd y}\m(z)\p_{z\to y}^2\sum_{x\lhd z,~x\neq y}\frac{\p_{z\to x}}{(1-\p_{z\to x})^2}+\\
&\quad\quad+\sum_{y_1\neq y_2\in P(0)}f(y_1)g(y_2)\sum_{z\rhd y_1,y_2}\m(z)\p_{z\to y_1}\p_{z\to y_2}\sum_{x\lhd z,~x\neq y_1,y_2}\frac{\p_{z\to x}}{(1-\p_{z\to x})^2}.
\end{align*}
If $P$ possesses Properties TL.1, TL.2, or equivalently, \eqref{eq:prop_1_for_norm_sqr_hat_f}  \eqref{eq:prop_2_for_norm_sqr_hat_f}, we obtain
\begin{equation*}
\sum_{x\in P(0)}\m(x)\langle \hf_x,\hg_x\rangle=
\csame \langle f,g\rangle + {\cdiff}\langle Af,g\rangle.
\end{equation*}
For the last item, recall that the adjacency matrix at the link of $x$
\begin{align}\label{eq:adj_matrix_link}
(A_xf)(z)&=\sum_{u\rhd z}\frac{\m_x(u)(\p_x)_{u\to z}}{\m_x(z)(1-(\p_x)_{u\to z})}\sum_{x\lhd w\lhd u,~w\neq z}(\p_x)_{u\to w}f(w)=\\
\notag&=\sum_{u\rhd z}\frac{\m(u)\p_{u\to z}}{\m(z)(\sum_{c\in \Ch(u\to x)}\p(c))(1-\frac{\p_{u\to z}\p_{z\to x}}{\sum_{c\in \Ch(u\to x)}\p(c)})}\sum_{x\lhd w\lhd u,~w\neq z}\p_{u\to w}\p_{w\to x}f(w).
\end{align}
Consider the expression $\sum_{x\in P(0)}\m(x)\langle A_x\hf_x,\hg_x\rangle_x,$ for $f,g\in C^0.$
\begin{align*}
\sum_{x\in P(0)}\m(x)&\langle A_x\hf_x,\hg_x\rangle_x=
\sum_{x\in P(0)}\m(x)\sum_{z\in P_x(0)}\m_x(z)
\sum_{u\rhd z}\frac{\m(u)\p_{u\to z}}{\m(z)(\sum_{c\in \Ch(u\to x)}\p(c))(1-\frac{\p_{u\to z}\p_{z\to x}}{\sum_{c\in \Ch(u\to x)}\p(c)})}\cdot\\&\quad\quad\quad\quad\quad\quad\quad\quad\quad\quad\quad\quad\quad\quad\quad\quad\quad\quad\quad\quad\quad\quad\quad\quad\cdot\sum_{x\lhd w\lhd u,~w\neq z}\p_{u\to w}\p_{w\to x}\hf_x(w)\hg_x(z)=\\
&=\sum_{x\in P(0)}\sum_{z\in P_x(0)}\m(z)\p_{z\to x}
\sum_{u\rhd z}\frac{\m(u)\p_{u\to z}}{\m(z)(\sum_{c\in \Ch(u\to x)}\p(c))(1-\frac{\p_{u\to z}\p_{z\to x}}{\sum_{c\in \Ch(u\to x)}\p(c)})}\cdot\\&\cdot\sum_{\substack{x\lhd w\lhd u,\\w\neq z}}\p_{u\to w}\p_{w\to x}
\left(\frac{1}{1-\p_{w\to x}}\sum_{y_1\neq x,~y\lhd w}\p_{w\to y_1}f(y_1)\right)
\left(\frac{1}{1-\p_{z\to x}}\sum_{y_2\neq x,~y\lhd z}\p_{z\to y_2}g(y_2)\right)=\\
&=\sum_{P(0)\ni y_1,y_2}f(y_1)g(y_2)\sum_{\substack{P(0)\ni x\neq y_1,y_2,\\
z\neq w\in P(1),\\z\rhd x,y_2,w\rhd x,y_1,\\
u\in P(2),~u\rhd z,w}}\frac{\m(u)\p_{u\to z}\p_{u\to w}\p_{w\to y_1}\p_{z\to y_2}\p_{w\to x}\p_{z\to x}}
{(\sum_{c\in \Ch(u\to x)}\p(c))(1-\frac{\p_{u\to z}\p_{z\to x}}{\sum_{c\in \Ch(u\to x)}\p(c)})(1-\p_{w\to x})(1-\p_{z\to x})}
\end{align*}
Thus, if $P$ possesses Properties TL.3, TL.4, or equivalently,
\eqref{eq:prop_1_for_loc_of_adj} and \eqref{eq:prop_2_for_loc_of_adj},
we obtain \eqref{eq:loc_dec_of_adj}.
\end{proof}
\subsection{A General Trickling Down theorem}

The following theorem is a general version for a "one step" trickling down.
\begin{thm}\label{thm:trickling_general}
Let $P$ be a standard graded weighted poset possessing Property TL.
Suppose that $P$ is locally connected, 
and that the non trivial eigenvalues of the adjacency matrix of the link $P_x$  lie in
$[\nu,\mu].$ Then any non trivial eigenvalue $\lambda$ of the adjacency matrix of $P$ satisfies
\[\frac{\csame\nu-\csametwo}{1-\nu}\leq\lambda\leq \frac{\csame\mu-\csametwo}{1-\mu}.\]
\end{thm}
\begin{proof}
Let $P$ be a standard, weighted graded poset possessing Property TL
with constants $\csametwo,\cdifftwo,\csame,\cdiff.$
We first prove, without the connectivity assumption, that
\begin{equation}\label{eq:trickling_upper_general}0\leq (1-\mu)\lambda^2+(\cdiff\mu-\cdifftwo)\lambda+\mu\csame-\csametwo\end{equation}
and
\begin{equation}\label{eq:trickling_lower_general}
0\geq (1-\nu)\lambda^2+(\cdiff\nu-\cdifftwo)\lambda+\nu\csame-\csametwo.\end{equation}
Let $f\perp\one$ be an eigenfunction of norm $1$ of $A$ for eigenvalue $\lambda.$ Then from Equation \eqref{eq:loc_dec_of_adj}  we obtain
\begin{equation}\label{eq:intermediate_for_trickling}
\lambda=\langle Af,f\rangle = \frac{1}{\cdifftwo}
\left(\sum_{x\in P(0)}\m(x)\langle A_x\hf_x,\hf_x\rangle_x-
\csametwo \right).
\end{equation}
Now, $\hf_x=\langle\hf_x,\one_x\rangle_x\one_x+\hf^\perp_x,$ where $\hf^\perp_x$ is perpendicular to $\one_x.$

If we assume that for all $x\in P(0)$ the non trivial eigenvalues of $A_x,$ are in $[\nu,\mu],$ we can write
\begin{equation*}
\langle A_x\hf_x,\hf_x\rangle_x=\langle\hf_x,\one_x\rangle_x^2\langle A_x\one_x,\one_x\rangle+\langle A_x\hf^\perp_x,\hf^\perp_x\rangle_x=\langle\hf_x,\one_x\rangle_x^2+\langle A_x\hf^\perp_x,\hf^\perp_x\rangle_x,
\end{equation*}
thus
\[\langle\hf_x,\one_x\rangle_x^2+\nu(\nr\hf_x\nr^2-\langle\hf_x,\one_x\rangle_x^2)\leq
\langle A_x\hf_x,\hf_x\rangle_x\leq \langle\hf_x,\one_x\rangle_x^2+\mu(\nr\hf_x\nr^2-\langle\hf_x,\one_x\rangle_x^2)\]
or
\begin{equation}\label{eq:ineq_A_xf_x}
(1-\nu)\langle\hf_x,\one_x\rangle_x^2+\nu\nr\hf_x\nr^2\leq
\langle A_x\hf_x,\hf_x\rangle_x\leq (1-\mu)\langle\hf_x,\one_x\rangle_x^2+\mu\nr\hf_x\nr^2.
\end{equation}
Plugging \eqref{eq:ineq_A_xf_x} into \eqref{eq:intermediate_for_trickling}, and using \eqref{eq:norm_sqr_hat_f}, Proposition \ref{prop:trickling_localization1},~\eqref{eq:f_x_1_x} we obtain
\begin{align*}\lambda &\leq
\frac{1}{\cdifftwo}
\left(\sum_{x\in P(0)}\m(x)((1-\mu)\langle\hf_x,\one_x\rangle_x^2+\mu\nr\hf_x\nr^2_x)-
\csametwo \right)=\\
&=\frac{1}{\cdifftwo}\left(
(1-\mu)\sum_{x\in P(0)}\m(x)(Af(x))^2+\mu(\csame\nr f\nr^2+{\cdiff}\langle Af,f\rangle)-
\csametwo \right)=\\
&=\frac{1}{\cdifftwo}\left(
(1-\mu)\nr Af\nr^2+\mu(\csame +\cdiff\lambda)-\csametwo \right)=
\frac{1}{\cdifftwo}\left((1-\mu)\lambda^2+\mu(\csame +\cdiff\lambda)-\csametwo \right).
\end{align*}
The same analysis can be applied for the lower bound, we summarize these inequalities
\begin{equation}\label{eq:trickling_general}
\frac{1}{\cdifftwo}\left((1-\nu)\lambda^2+\nu(\csame +{\cdiff}\lambda)-\csametwo \right)
\leq\lambda\leq
\frac{1}{\cdifftwo}\left((1-\mu)\lambda^2+\mu(\csame +{\cdiff}\lambda)-\csametwo \right),
\end{equation}
which is equivalent to \eqref{eq:trickling_upper_general}, \eqref{eq:trickling_lower_general}.

We now add the connectivity assumption. Write
\[q_\eta(x)=(1-\eta)x^2+(\cdiff\eta-\cdifftwo)x+\csame\eta-\csametwo\]
$q_\eta$ is a quadratic function of $x$ whose leading coefficient is positive when $\eta<1.$ In this case the domain in which it is non positive is the interval between its two roots.
Using Lemma \ref{lem:cdiff_plus_csame}, 
$q_\eta(1)= 0.$ 
The second root of $q_\eta$ is $\frac{\csame\eta-\csametwo}{1-\eta}.$
Substituting $\eta=\mu,$ which is smaller than $1,$ by the connectivity assumption on the links, and using \eqref{eq:trickling_upper_general}, we see that either $\lambda$ is at least the maximal root, or at most the minimal root. The former is impossible, since $\lambda<1,$ by connectivity of $P$. Thus, \[\lambda\leq \frac{\csame\mu-\csametwo}{1-\mu}.\]
A similar argument using \eqref{eq:trickling_lower_general} shows $ \frac{\csame\nu-\csametwo}{1-\nu}\leq\lambda.$
\end{proof}
Theorem \ref{thm:trickling_general} and Lemma \ref{lem:xzw_implies_decomp_Ahf} immediately give
\begin{thm}\label{thm:trickling_structural}
Let $P$ be a standard graded poset. Suppose that $P$ is $2-$skeleton regular, with constants $\Nlow_1,\Nlow_2,$ $\Nmid_1$ and $\Rwye.$
Assume also that $P$ and any link $P_x,$ for $x\in P(0)$ are connected, and that the non trivial eigenvalues of the adjacency matrix of $P_x$  lie in $[\nu,\mu]$.
Then
\[\frac{\frac{\nu}{\Nlow_1-1}-\frac{({\Nlow_2\Nlow_1-\Nmid_1})\Rwye(\Rwye-1)}
{(\Nmid_1-1)(\Nmid_1)^2({\Nlow_1}-1)^2}}{1-\nu}\leq\lambda\leq\frac{\frac{\mu}{\Nlow_1-1}-\frac{({\Nlow_2\Nlow_1-\Nmid_1})\Rwye(\Rwye-1)}
{(\Nmid_1-1)(\Nmid_1)^2({\Nlow_1}-1)^2}}{1-\mu}.\]
Let $P$ be a locally connected, locally $2-$skeleton regular of rank $d.$ Suppose that for any $x\in P(d-2)$ the non trivial eigenvalues of the adjacency matrix of the link $P_x$  lie in $[\nu,\mu]$.
Then for any $-1\leq i\leq d-3,$ any $s\in P(i),$ every non trivial eigenvalue of the adjacency matrix of $P_x$ $\lambda$
satisfies
\begin{align*}T(\cdots T(T(\nu;C_{d-3} , B_{d-3} &);C_{d-4} , B_{d-4} );\cdots); C_{i}, B_{i})\leq\lambda\leq\\& \leq T(\cdots T( T(\mu;C_{d-3} , B_{d-3} );C_{d-4} , B_{d-4} );\cdots); C_{i}, B_{i}),\end{align*}
where $T(x; C,B)=\frac{Cx-B}{1-x}$ , $C_{j}=\frac{1}{\Nlow_{j,1}-1},~~B_{j}=\frac{({\Nlow_{j,2}\Nlow_{j,1}-\Nmid_{j,1}})\Rwye_j(\Rwye_j-1)}
{(\Nmid_{j,1}-1)(\Nmid_{j,1})^2({\Nlow_{j,1}}-1)^2},$
and $\Nlow_{j,1},\Nmid_{j,1},\Rwye_{j}$ are the structure constants provided in the definition of $2-$skeleton regularity.
\end{thm}
We note that the iterative version of the theorem generalizes to the setting of more general weighted graded posets which are locally connected and possess a \emph{local TL property}.

Understanding the fixed points of the transformation $T$ is useful for understanding the behaviour of the spectral gaps in lower links, especially when the trickling down is performed repeatedly. We shall examine this in the Grassmannian example below, and use this idea when analyzing the expansion of the Grassmannian poset which we construct in Section \ref{sec:posetification}.
\begin{ex}\label{ex:trickling}
When $P$ is a simplicial complex with a standard weight scheme, then \[\csame=1,~\cdiff=0,~\csametwo=0,~\cdifftwo=1.\]
Thus,
we obtain the following upper bound on the second eigenvalue of the adjacency matrix:
\begin{equation}\label{eq:ineq_upper_simplices}
\lambda\leq \frac{\mu}{1-\mu}.
\end{equation}
Similarly, we obtain the lower bound
\begin{equation}\label{eq:lower_upper_simplices}
\lambda\geq \frac{\nu}{1-\nu}
\end{equation}
on the least eigenvalue, reproducing the result of \cite{O}.

Moving to the Grassmannian poset, this time
\[\csame=\frac{1}{q},~\cdiff=\frac{q-1}{q},~\csametwo=0,~\cdifftwo=1.\]
Thus,
we obtain the following upper bound
on the second eigenvalue of the adjacency matrix:
\begin{equation}\label{eq:ineq_upper_grassm}
\lambda\leq \frac{\mu}{q(1-\mu)}
\end{equation}
Similarly, we obtain a lower bound
\begin{equation}\label{eq:ineq_upper_grassm}
\lambda\geq \frac{\nu}{q(1-\nu)}
\end{equation}on the lowest eigenvalue.
In the Grassmannian case $0, (q-1)/q$ are fixed points of the transformation $T.$ Zero is attracting and $(q-1)/q$ repulsing. Thus, if $\mu$ is even slightly smaller than $(q-1)/q,$ the upper bounds on the second eigenvalues of links get better as we go down the rank. In this sense $(q-1)/q$ is \emph{critical}. The lower bound on negative eigenvalues also tends to $0$ fast. 

\end{ex}

Theorems \ref{thm:trickling_general}, \ref{thm:trickling_structural} allow inducing the spectral bounds of the top links to the whole poset.


\section{Posetification and constructions of sparse expanding posets}\label{sec:posetification}
Let $S$ be a set, and $X$ a set of subsets of $S$, which includes the empty set, and is closed under taking subsets. Such sets $X$ are in natural bijection with simplicial complexes on the vertex set $S,$ and the poset structure on $X$ is the containment order.
Suppose we are given a poset $P$ and an order preserving association $I\mapsto P_I\in P,~I\in X.$ Write $P_X$ for the \emph{posetification of $X$} which is the subposet of $P$ defined by
\[P_X=\{y\in P|\exists I\in X~\text{s.t. }y\leq P_I\}.\]
When $X$ is endowed with a standard weight scheme the posetification is also naturally endowed by a standard scheme, defined by putting $\m(V_I)=\m(I),$ for any maximal simplex $I.$

It turns out that when the simplicial complex $X$ is an expander, then $P_X$ inherits expansion properties which depend on the expansion properties of $X$ and of $P.$

We illustrate this idea in the following theorem, which, to the best of our knowledge, is the first example of a sparse expanding Grassmannian poset.

\begin{thm}\label{thm:expanding_grassmannian}
Let $X$ be a local spectral expander of dimension $d,$ on the vertex set $[n].$ Suppose that $X$ is locally connected\footnote{we remind the reader that 'locally connected' means that each link $P_x$ for $x\in P(\leq d-2),$ including $P=P_\smallest,$ is connected.}, and the links of all its $d-2$-dimensional cells are regular bipartite expanders, with a second eigenvalue bounded by $\epsilon.$
Let $V$ be a vector space over $\F_q$ with basis $e_1,\ldots,e_n,$ and to each $I\subseteq [n]$ associate the vector space \[V_I=\text{span}\{v_i|i\in I\}\subseteq V.\]
Then the posetification $V_X$ is a Grassmannian poset which satisfies:
\begin{enumerate}
\item It is a subposet of the Grassmannian poset on $V.$
\item Suppose that each cell of $X$ of dimension $k,$ for any $k\geq 0,$ is contained in at most $Q$ cells of dimension $k+1.$ Then for every element of $V_X$ of rank $k\geq 0,$ the number of elements of rank $k+1$ which cover it is upper bounded by a function of $q,Q$ and $d$ (independent of $n$).
\item The second largest eigenvalue of each link is at most $\frac{q-1}{q}+f(\epsilon,d,q),$ where $f$ is a function which tends to $0$ as $\epsilon\to 0,$ while the lowest eigenvalue is at least $-\frac{1}{q}.$
\end{enumerate}
\end{thm}
\begin{rmk}\label{rmk:further_props_of_V_X}
We can say more about the expansion of $V_X.$
\begin{enumerate}
\item By Theorem \ref{thm:equiv_for_2_sided} $V_X$ is also a $\frac{q-1}{q}+f(\epsilon,d,q)-$global eposet.
\item It follows from the proof below that as a function of $\epsilon,$ for small $\epsilon,$ $f$ is bounded by a linear function.
Then, by Example \ref{ex:AL_Grass} with $\lambda_i=1-\frac{1}{q}+O(\epsilon)$ for each $i,$ we learn that for all $k$ the second eigenvalue of $M^+_k$ is at most $1-q^{-k-1} + O(\epsilon)$ for large $q.$
\end{enumerate}
\end{rmk}
\begin{proof}
The first two items are clear. For the third, we first prove that for any $x\in (V_X)_{d-2},$ all non trivial eigenvalues of the link adjacency matrix $A_x$ lie in the interval $[-\frac{1}{q},\frac{q-1+\epsilon}{q}].$

There are three cases.
The first is when $x$ is not contained in any $V_I$ with $|I|=d,$ hence it is contained in a single $V_I$ for $|I|=d+1.$
The second is that $x\subset V_I$ for $I\subseteq[n],~|I|=d$, but $x\neq V_J$ for any $J\subset[n]$ with $|J|=d-1.$
The last case is that $x=V_I$ for some subset $I\subseteq [n],~|I|=d-1$.

The first case is the easiest. In this case the link of $x$ is equivalent to the Grassmannian poset of $\F_q^2,$ meaning the $((V_X)_x)(0)$ are the $q+1$ lines in a vector space isomorphic to $\F_q^2.$ The corresponding adjacency matrix is the normalized adjacency matrix of $K_{q+1},$ the complete graph on $q+1$ elements, and its eigenvalues are $1,$ with multiplicity $1,$ and $-\frac{1}{q}$ with multiplicity $q.$

Turning to the second case, suppose $x$ is contained in precisely $p$ $d+1$-spaces $V_{I_1},\ldots, V_{I_p}$ with $\bigcap_h I_h=I.$
The link of $x$ is equivalent to the subposet of $\F_q\{e_0,\ldots,e_p\},$ where $\{e_0,\ldots,e_p\}$ are independent generators, with \[((V_X)_x)_1=\{\text{span}(e_0,e_1),\text{span}(e_0,e_2),\ldots,\text{span}(e_0,e_p)\},\]
and $((V_X)_x)(0)$ are all the lines contained in these $2-$spaces.

The adjacency matrix is the normalized adjacency matrix of a bouquet of $p$ $K_{q+1}.$ This is a $(pq+1)\times(pq+1)$ matrix with $0$ diagonal, all non diagonal entries in the last row are $\frac{1}{pq},$ while all non diagonal entries in the last column are $\frac{1}{q}.$ For any different $0\leq i,j< pq,$ the $(i,j)$ entry is $0$ if $\lfloor\frac{i}{q}\rfloor\neq\lfloor\frac{j}{q}\rfloor,$ and otherwise it is $\frac{1}{q}.$
This matrix has the constant vector as the eigenvector of $1.$

It has two other families of eigenvectors:
$v_{a,b},$~$a=0,\ldots,p-1,$~$b\in[q-1],$ all of whose entries are $0,$ except the $aq$ entry which is $1$ and the $aq+b$ entry which is $-1.$ These are eigenvectors for $-\frac{1}{q}.$
The other family $v_i,~i\in[p-1]$ is a family of eigenvectors for $\frac{q-1}{q}.$ All entries of $v_i$ are $0,$ except the first $q$ which are $1$ and the entries $iq,iq+1,\ldots,iq+q-1$ which are $-1.$
By considering the support of the vectors in each family it is straight forward to verify that they are linearly independent. These families span two orthogonal spaces, which are also orthogonal to the constant vectors. Thus, these families, and the eigenvector of $1$ amount to $pq$ independent eigenvalues. Their sum is
\[1+p(q-1)\frac{-1}{q}+(p-1)\frac{q-1}{q}=\frac{1}{q}.\] Trace considerations show that the remaining eigenvalue is also $-\frac{1}{q}.$

We now turn to the last case. Now $x$ is a vector space of dimension $d-1.$ It corresponds to a $d-2$ cell of $X.$ Let $G=(V,E)$ be the link of this cell, which is a graph. By assumptions it is bipartite, connected, and the second eigenvalue is at most $\epsilon.$
Similarly to the previous cases, the adjacency matrix of the link of $x$ is the normalized adjacency matrix of the graph $G'$ obtained from $G$ as follows The vertex set of $G'$ is $V\cup (E\times[q-1]).$ There is an edge between $v,u\in V$ if they were connected in $G,$ there is an edge between and $(e,i),(e,j),~i,j\in[q-1],$ and there is an edge between $v,(e,i)$ if $e$ is an edge of $v$ in $G.$
$G'$ is a connected graph on $m+(q-1)|E|$ vertices, where $m=|V|.$
We now write three families of eigenfunctions for the normalized adjacency matrix of $G'.$
\begin{enumerate}
\item $H_1(G,\mathbb{Z}),$ the first homology of $G$, is of rank $|E|-m+1,$ and is generated by simple cycles of even length in $G,$ by bipartiteness. Let $C_1,\ldots, C_{|E|-m+1}$ be simple cycles representing such generators.
    Simple induction on the construction of the cycles allows assuming that each $C_i$ contains an edge $e_i$ such that $e_i\notin C_j$ for $j<i.$
    For each $C_i$ order its edges starting from $e_i$ in any way which agrees with one of its two cyclic orders. Let $c_i$ to be the vector whose $(e,j)$ entries are $1,$ if $e$ is an edge at an even place in $C_i,$ according to this order, they are $-1$ if $e$ is located in an odd place, and all other entries are $0.$ These vectors are eigenfunctions for $\frac{q-2}{q}.$
\item For every $e\in E,~i\in[q-2]$ set $f_{e,i}$ to be the vector all of whose entries are $0,$ but the $(e,i)$ and $(e,q-1)$ which are $1,-1$ respectively. These $(q-2)|E|$ vectors are eigenfunctions for $\frac{-1}{q}.$
\item Let $g_1,\ldots, g_m$ be the eigenfunctions of $G,$ ordered according to the order $1=\lambda_1>\lambda_2\geq\ldots>\lambda_m=-1,$ of the corresponding eigenvalues.
    Define $g_m^a\in \R^{V\cup (E\times[q-1])}$ by
    \[g_k^a|_V=g,~~g_k^a(e,i)=a(g_k(v)+g_k(u)),~~\text{for }e=\{u,v\}.\]
    Applying the adjacency operator of $G'$ to $g_k^a$ we obtain
    \[(A_{G'}g_k^a)(v)=(\frac{\lambda_k}{q}+ \frac{a}{q}(q-1)(1+\lambda_k))g_k^a(v),\]\[
    (A_{G'}g_k^a)(e,i)=(\frac{1}{q}+\frac{a(q-2)}{q})(g_k(v)+g_k(u)),~~\text{for }e=\{u,v\}.\]
    Note that $g_m$ assigns opposite values to neighboring vertices, hence $g_m^a$ is independent of $a$ and it is easy to see that its eigenvalue is $\frac{-1}{q}.$ We write $g_m^-$ for $g_m^a.$
    In order for $g_k^a,~k\neq m,$ to be an eigenfunction we must have
    \begin{equation}\label{eq:eigen}
    (\frac{\lambda_k}{q}+ \frac{a}{q}(q-1)(1+\lambda_k))a=\frac{1}{q}+\frac{a(q-2)}{q},\end{equation}
    and in this case the eigenvalue will be $\frac{\lambda_k}{q}+ \frac{a}{q}(q-1)(1+\lambda_k).$
    \eqref{eq:eigen} is equivalent to
    \[a^2(q-1)(1+\lambda_k)+a(\lambda_k-q+2)-1=0.\]This quadratic equation has two different solutions, $a=\frac{1}{1+\lambda_k},$ and $a=-\frac{1}{q-1}.$ We denote by $g^+_k$ the eigenfunction for the former solution, and by $g^-_k$ the eigenfunction for the latter. The eigenvalue of $g^+_k$ is $\lambda^+_k=\frac{q-1+\lambda_k}{q},$ while the eigenvalue of $g^-_k$ is $\lambda^-_k=-\frac{1}{q}.$
\end{enumerate}
Altogether we wrote $2m-1+|E|-m+1+|E|(q-2)=(q-1)|E|+m.$ We claim that these vectors are a complete set of eigenfunctions. Since they are of the correct cardinality, it is enough to prove linear independence. In addition, it is enough to prove the linear independence inside the sets of eigenfunctions for the same eigenvalue.

Regarding the eigenvalue $-\frac{1}{q},$ the vectors $g_1^-,\ldots,g_m^-$ are linearly independent, as can be seen from their restriction to the $V-$entries, and using the independence of $g_1,\ldots,g_{m}.$ They are orthogonal to the elements $f_{e,i},~e\in E,~i\in[q-2],$ hence independent of them as well. The latter vectors are also easily seen to be independent.

Regarding the eigenvalue $\frac{q-2}{q},$ the elements $c_i$ are linearly independent.
Indeed, assume towards contradiction that $\sum a_ic_i=0,$ for some scalars $a_i,$ not all are $0.$ Let $i^*$ be the maximal index with a non zero coefficient. When evaluating on $(e_{i^*},1)$ we obtain, on the one hand $(\sum a_ic_i)(e_{i^*},1)=0.$ On the other hand, by the choice of the edges $e_i,$ $(\sum a_ic_i)(e_{i^*},1)=a_{i^*}c_{i^*}(e_{i^*},1)=\pm a_{i^*},$ which is a contradiction. Since no $\lambda_k=-1,$ for $k\neq m,$ no $\lambda_k^\pm=\frac{q-2}{q},$ so there are no more $\frac{q-2}{q}$ eigenfunctions.

We are left only with $g^+_k,~k\in[m-1].$ These have eigenvalues different from the previous ones we considered, and again they can be seen to be linearly independent by restricting to $V$ and using the independence of $g_1,\ldots,g_{m-1}.$

Note that for any $w\in[-\frac{1}{q},0],$ the map $w\mapsto Tw=\frac{w}{q(1-w)}$ satisfies
$Tw\in[w,0].$
Similarly, $T$ maps $[0,\frac{q-1}{q}]$ onto itself. For $w>\frac{q-1}{q},$ $Tw>w,$ but, since $\frac{q-1}{q}$ is a fixed point of $T,$ one has \[T^{d-2}(\frac{q-1+\epsilon}{q})=\frac{q-1}{q}+f(\epsilon,d,q),\]where, for fixed $d,q,$ $f\to0$ as $\epsilon\to 0,$ by the continuity of $T.$

Thus, by applying Theorem \ref{thm:trickling_structural}, with the parameters of the Grassmannian, as in Example \ref{ex:trickling}, all non trivial eigenvalues of the adjacency matrix of each link are seen to be at most $\frac{q-1}{q}+f(\epsilon,d,q).$
\end{proof}

The Ramanujan complexes constructed in \cite{LSV1,LSV2} are $d-$dimensional simplicial complexes which are obtained as arithmetic quotients of Bruhat-Tits buildings of dimension $d.$ They are locally connected, their $d-2-$dimensional links are bipartite graphs and, and all non trivial eigenvalues of the adjacency matrices of the links are bounded from above by $O(\frac{1}{\sqrt{Q}}),$ where $Q,$ the \emph{thickness} of the Ramanujan complex is the number of top cells touched by a given codimension $1$ cell, minus $1.$ For any $Q_0$ and given $d$ there exist $d$ dimensional Ramanujan complexes with $Q\geq Q_0.$ 

\begin{cor}\label{cor:contruction}
Let $q$ be a prime power, and let $X$ be a $d-$dimensional Ramanujan complex of thickness $Q,$ large enough as a function of $q,d.$ Then the posetification $V_X$ is a two sided bounded degree expanding Grassmannian poset. The second largest eigenvalue of each link is at most $\frac{q-1}{q}+o(1),$ while the lowest eigenvalue is at least $-\frac{1}{q}.$
\end{cor}
In \cite{DDFH} it was conjectured that there exist high dimensional bounded degree expanding Grassmannian posets for any bound $\mu>0.$ This corollary proves this conjecture for $\mu=\frac{q-1}{q}+o(1).$

The question for general $\mu$ still awaits for an answer. Based on the trickling down for Grassmannians, if one can find a bounded degree Grassmannian complex of high enough dimension such that the second eigenvalue of the top links is bounded from above by $\lambda<\frac{q-1}{q}-c,$ for some positive $c$, lower skeletons will be $\mu-$local spectral Grassmannian expanders, for arbitrary small $\mu$ (there is a trade off between $c,~\mu,$ the dimension and how low one should go).

Regarding how large must $Q$ be, observe that
the derivative of the transformation $x\mapsto Tx,$ from the end of the previous proof, at $(q-1)/q$ is $q$. Thus, for $\epsilon$ small enough, as a function of $q,d,$
\[T^{d-2}(\frac{q-1+\epsilon}{q})=\frac{q-1}{q}+q^{d-2}\epsilon +o(\epsilon).\] Since for Ramanujan complexes $\epsilon=O(\frac{1}{\sqrt{Q}})$ in order to obtain expansion we need $Q\gg q^{2d-4}.$

\section{Posets with approximate localization properties}\label{sec:approx}
Regular posets are very rigid objects. Assumptions UL, AL and TL allow more flexibility, but since they are defined by many equations, they are also not very flexible. The purpose of this section is to briefly describe how our methods extend to the much more general setting in which instead of requiring the above assumptions to hold, we require them to hold approximately. We claim that in this case our theorems will also work approximately, where the error terms in the theorem are explicitly determined by quality of the approximation in the assumptions. We will illustrate this for the UL Assumption and its consequences, but similar approximate versions exist for the other assumptions and applications. In the illustration below we do not try to optimize the error terms.
\begin{definition}\label{def:ass_I_approx}
Let $(P,\leq,\rk,\m,\p)$ be a weighted graded poset of rank $d.$ We say that $P$ possesses the Approximate UL Property if
\begin{enumerate}
\item
There are constants $\cxyz_0,\ldots,\cxyz_{d-1},$ and $\epsxyz_0,\ldots,\epsxyz_{d-1}$ such that for all $y\in P(l),~0\leq l<d,~x\rhd y:$\[\sum_{z\in P(l-1),~z~\lhd y}\frac{\p^2_{y\to z}}{\sum_{x\in\Ch(x\to z)}\p(c)}\in[\cxyz_l-\epsxyz_l,\cxyz_l+\epsxyz_l].\]
\item
There exist constants $\cdia_0,\ldots,\cdia_{d-1},\epsdia_0,\ldots,\epsdia_{d-1}$ such that for all $y_1\neq y_2\in P(l),~0\leq l<d,$ which are both covered by at least one element, and any $x$ which covers both,
\[\sum_{z\in P(l-1),~z\lhd y_1,y_2}\frac{\p_{y_1\to z}\p_{y_2\to z}}{\sum_{x\in\Ch(x\to z)}\p(c)}\in[\cdia_l-\epsdia_l,\cdia_l+\epsdia_l].\]
\item
There exist constants $\csqr_0,\ldots,\csqr_{d-1}, \epsqr_0,\ldots,\epsqr_{d-1}$ such that for all $y\in P(l),~0\leq l<d,$
\[\sum_{x\rhd y}\frac{\m(x)\p^2_{x\to y}}{\m(y)}\in[\csqr_l-\epsqr_l,\csqr_l+\epsqr_l].\]
\end{enumerate}
\end{definition}
When all constants $\epsdia_k,\epsxyz_k,\epsqr_k$ are $0$ we obtain the usual UL Property.
When $P$ possesses the Approximate UL Property, the consequences of this property also hold approximately:
\begin{prop}\label{prop:garland_posets_2_approx}
Let $P$ be a graded weighted poset of rank $d,$ which possesses the Approximate UL Property with constants $(\cxyz_i)_i,(\csqr_i)_i,(\cdia_i)_i,(\epsxyz_i)_i,(\epsqr_i)_i,(\epsdia_i)_i$. Then for any $0\leq l\leq d-1$ and $f,g\in C^l$
\begin{align*}\bigg|\langle U_lf,U_lg\rangle-&
\frac{1}{\cdia_l}\sum_{z\in P(l-1)}\m(z)\langle U_{z,0}(f_z),U_{z,0}(g_z)\rangle_z+
\csqr_l(\frac{\cxyz_l}{\cdia_l}-1)\langle f,g\rangle\bigg|\\&\leq
\frac{\epsdia_l}{\cdia_l}\langle U_l|f|,U_l|g|\rangle +\frac{1}{\cdia_l}\left(\csqr_l(\epsxyz_l-\epsdia_l)+\epsqr_l(\cxyz_l-\cdia_l)+\epsqr_l|\epsxyz_l-\epsdia_l|\right)\langle |f|,|g|\rangle.\end{align*}
\end{prop}
\begin{proof}
We prove
\begin{align*}\cdia_l\langle U_lf,U_lg\rangle-&
\sum_{z\in P(l-1)}\m(z)\langle U_{z,0}(f_z),U_{z,0}(g_z)\rangle_z+
\csqr_l({\cxyz_l}-{\cdia_l})\langle f,g\rangle\\&\leq
\epsdia_l\langle U_l|f|,U_l|g|\rangle +\left(\csqr_l(\epsxyz_l-\epsdia_l)+\epsqr_l(\cxyz_l-\cdia_l)+\epsqr_l|\epsxyz_l-\epsdia_l|\right)\langle |f|,|g|\rangle\end{align*}
the lower bound is proven similarly.
We start as in the derivation of \eqref{eq:U_iU_i2}
\begin{align*}
\sum_{z\in P(l-1)}&\m(z)\langle U_{z,0}(f_z),U_{z,0}0(g_z)\rangle_z
\\
&=\sum_{y_1,y_2\in P(l)}f(y_1)g(y_2)\sum_{x\rhd y_1,y_2}\m(x)\p_{x\to y_1}\p_{x\to y_2}\sum_{z\lhd y_1,y_2}\frac{\p_{y_1\to z}\p_{y_2\to z}}{\sum_{c\in\Ch(x\to z)}\p(c)}\\
&\leq\cdia_l\sum_{y_1\neq y_2\in P(l)}f(y_1)g(y_2)\sum_{x\rhd y_1,y_2}\m(x)\p_{x\to y_1}\p_{x\to y_2}\\&+\epsdia_l\sum_{y_1\neq y_2\in P(l)}|f(y_1)g(y_2)|\sum_{x\rhd y_1,y_2}\m(x)\p_{x\to y_1}\p_{x\to y_2}\\
&
+\cxyz_l\sum_{y\in P(l)}f(y)g(y)\sum_{x\rhd y}\m(x)\p^2_{x\to y}+\epsxyz_l\sum_{y\in P(l)}|f(y)g(y)|\sum_{x\rhd y}\m(x)\p^2_{x\to y}\end{align*}
\begin{align*}
&=\cdia_l\sum_{y_1,y_2\in P(l)}f(y_1)g(y_2)\sum_{x\rhd y_1,y_2}\m(x)\p_{x\to y_1}\p_{x\to y_2}
+(\cxyz_l-\cdia_l)\sum_{y\in P(l)}f(y)g(y)\sum_{x\rhd y}\m(x)\p^2_{x\to y}\\
&+
\epsdia_l\sum_{y_1,y_2\in P(l)}|f(y_1)g(y_2)|\sum_{x\rhd y_1,y_2}\m(x)\p_{x\to y_1}\p_{x\to y_2}
+(\epsxyz_l-\epsdia_l)\sum_{y\in P(l)}|f(y)g(y)|\sum_{x\rhd y}\m(x)\p^2_{x\to y}
\\
&\leq\cdia_l\langle U_l f, U_lg\rangle+\csqr_l(\cxyz_l-\cdia_l)\sum_{y\in P(l)}f(y)g(y)\m(y)+\epsqr_l(\cxyz_l-\cdia_l)\sum_{y\in P(l)}|f(y)g(y)|\m(y)\\
&+\epsdia_l\langle U_l|f|,U_l|g|\rangle+\csqr_l(\epsxyz_l-\epsdia_l)\sum_{y\in P(l)}|f(y)g(y)|\m(y)+
\epsqr_l|\epsxyz_l-\epsdia_l|\sum_{y\in P(l)}|f(y)g(y)|\m(y)
\\&=\cdia_l\langle U_lf,U_lg\rangle+\csqr_l(\cxyz_l-\cdia_l)\langle f,g\rangle+\epsdia_l\langle U_l|f|,U_l|g|\rangle +\left(\csqr_l(\epsxyz_l-\epsdia_l)+\epsqr_l(\cxyz_l-\cdia_l)+\epsqr_l|\epsxyz_l-\epsdia_l|\right)\langle |f|,|g|\rangle,
\end{align*}
where we have used the definition of the approximated UL property and \eqref{eq:U_iU_i1}.
\end{proof}
For simplicity in what follows we assume $\epsdia_l=\epsqr_l=\epsxyz_l=\epsilon_l,$ and then the error term of Proposition \ref{prop:garland_posets_2_approx} becomes
\[\frac{\epsilon_l}{\cdia_l}\langle U_l|f|,U_l|g|\rangle +\epsilon_l(\cxyz_l-\cdia_l)\langle |f|,|g|\rangle\leq\epsilon_l\frac{1+\cxyz_l-\cdia_l}{\cdia_l}\nr f\nr\nr g\nr,\]
by Cauchy-Schwarz and the fact that $U_l$ has all eigenvalues at most $1.$

We now move to the approximated version of Proposition \ref{prop:towards UD-DU}. Repeating the proof of that Proposition, and plugging in our error estimate of Proposition \ref{prop:garland_posets_2_approx} in the last step of the proof (where the summation over $\langle U_{x,0}f_x,U_{x,0}g_x\rangle_x$ is taken) yields
\begin{prop}\label{prop:towards UD-DU_approx}
If $P$ possesses the Approximate UL Property, then $\forall\alpha\in\R$
\begin{align*}\bigg|\cdia_l\langle U_l f, U_lg\rangle&-
\bigg(
(1-\alpha)\langle D_l f, D_lg\rangle +(\alpha -\csqr_l(\cxyz_l-\cdia_l))\langle f,g\rangle+
\\&+\sum_{x\in P(l-1)}\m(x)\langle f_x,(M^+_{x,0}-\alpha Id)(Id- M^-_{x,0})g_x\rangle_x\bigg)\bigg|\leq
\epsilon_l(1+\cxyz_l-\cdia_l)\nr f\nr\nr g\nr,\end{align*}
\end{prop}
The approximated version of Proposition \ref{prop:KO_5_2} is
\begin{prop}\label{prop:KO_5_2_approx}
Let $P$ be a graded weighted poset of rank $d$ possessing the Approximated UL Property. Let $\{\alpha_j\}_{0\leq j \leq d-1}$ be real numbers. 
Then for every $0\leq k\leq d-1,~f\in C^k_0$ there exist elements $g_j,h_j\in C_0^j,~0\leq j\leq k$ such that
\begin{enumerate}
\item $f=g_k.$
\item For all $j\leq k$,~$\nr g_j\nr^2=\sum_{i=0}^j\nr h_i\nr^2.$
\item For all $j\leq k$ the difference\begin{align*}
\bigg|\nr U_j g_j\nr^2-\sum_{i=0}^ja_{j,i}\nr h_i\nr^2-\sum_{i=0}^jb_{j,i}\sum_{x\in P(i-1)}\m(x)\langle (g_i)_x,(M^+_{x,0}-\alpha_i Id)&(Id- M^-_{x,0})(g_i)_x\rangle_x\bigg|\end{align*}
is bounded by $(\sum_{i=0}^je_{j,i}\epsilon_i)\nr g_j\nr^2
,$ where $a_{j,i}, b_{j,i}$ are as in Proposition \ref{prop:KO_5_2},
and \begin{equation}\label{eq:e_i_j}e_{j,i}=e_{j,i}(\vec{\alpha})=\frac{1+\cxyz_i-\cdia_i}{\cdia_i}\prod_{h=i+1}^j\frac{1-\alpha_h}{\cdia_h}.\end{equation}
\end{enumerate}
\end{prop}
The proof is identical to that of Proposition \ref{prop:KO_5_2}, with the same choice of $h_i,g_i$ only that we use Proposition \ref{prop:towards UD-DU_approx} instead of Proposition \ref{prop:towards UD-DU}. By doing that it is straight forward to see that the error terms for $g_j$ are bounded by \[\sum_{i\leq j} e_{j,i}\epsilon_i\nr g_i\nr^2\leq (\sum_{i\leq j} e_{j,i}\epsilon_i)\nr g_j\nr^2,\]where $e_{j,i}$ are defined recursively by $e_{j,j}=\frac{1+\cxyz_j-\cdia_j}{\cdia_j},~e_{j,i}=\frac{1-\alpha_j}{\cdia_j}e_{j-1,i},$ or equivalently by \eqref{eq:e_i_j}.

The final result, whose derivation is identical to that of Theorem \ref{thm:5_3_5_4}, is
\begin{thm}\label{thm:5_3_5_4_approx}
Let $P$ be a graded weighted poset of rank $d$ possessing the Approximated UP Property.
Let $\alpha_j,~0\leq j \leq d-1$ be real numbers. 
Then for every $k\geq 0,~f\in C^k_0$ there exist elements $h_j\in C_0^j,~0\leq j\leq k$ such that
\[\nr f\nr^2=\sum_{j=0}^k\nr h_j\nr^2,\]
\[\nr U_k f\nr^2\leq \sum_{j=0}^k\left(a_{k,j}+ \sum_{i=j}^k b_{k,i} (\mu'_{i-1}-\alpha_i)_+\right)\nr h_j\nr^2+(\sum_{i=0}^ke_{k,i}\epsilon_i)\nr f\nr^2,\]
where $a_{k,j}(\vec{\alpha}),b_{k,j}(\vec{\alpha})$ are as in Proposition \ref{prop:KO_5_2}, and $e_{k,j}$ are given by \eqref{eq:e_i_j}.\end{thm}For the error term we use that $\nr g_j\nr\leq \nr f\nr$ for all $j$. Specific choices of $\alpha_j,$ give rise to explicit bounds, as in Section \ref{sec:Garland}.


%

\end{document}